\documentclass[12pt,twoside]{article}
\usepackage[left=1.5cm,right=1.5cm,top=3cm,bottom=2cm]{geometry}
\usepackage{paralist}
\usepackage{marginnote}
\usepackage{amsfonts}
\usepackage{comment}
\usepackage{enumitem}
\usepackage{tikz}
\usetikzlibrary{matrix}
\usepackage{amsthm}
\usepackage{amsrefs}
\usepackage{amssymb}
\usepackage{amsmath}
\usepackage{graphicx}
\usepackage{varwidth}
\usepackage{hyperref}
\usepackage{fancyhdr}
\usepackage{stmaryrd}
\usepackage{multirow}

\usepackage{blkarray}
\usepackage[autostyle]{csquotes}
\usepackage[mathscr]{euscript}
\hypersetup{
    colorlinks=true,
    linkcolor=black,
    filecolor=magenta,      
    urlcolor=black,
    citecolor=black
}
\DeclareMathAlphabet{\mathcalligra}{T1}{calligra}{c}{h}
\providecommand{\U}[1]{\protect\rule{.1in}{.1in}}
\newtheorem{theorem}{Theorem}[section]
\newtheorem{proposition}[theorem]{Proposition}
\newtheorem{lemma}[theorem]{Lemma}

\newtheorem{corollary}[theorem]{Corollary}

\newtheorem{remark}[theorem]{Remark}
\let\oldremark\remark
\renewcommand{\remark}{\oldremark\normalfont}

\newtheorem{example}[theorem]{Example}
\let\oldexample\example
\renewcommand{\example}{\oldexample\normalfont}

\def\<{{\langle}}
\def\>{{\rangle}}

\def\bea{\begin{eqnarray*} }
\def\eea{\end{eqnarray*} }
\def\be{\begin{equation} }
\def\ee{\end{equation} }

\def\qed{\ifhmode\unskip\nobreak\fi\ifmmode\ifinner
\else\hskip5 pt \fi\fi\hbox{\hskip5 pt \vrule width4 pt
height6 pt  depth1.5 pt \hskip 1pt }}

\DeclareMathOperator*{\Vol}{Vol}
\DeclareMathOperator*{\Area}{Area}

\DeclareMathOperator*{\supp}{supp}

\DeclareMathOperator*{\expo}{exp}
\DeclareMathOperator*{\Cut}{Cut}

\DeclareMathOperator*{\ess}{ess}
\DeclareMathOperator*{\ap}{ap}
\DeclareMathOperator*{\ad}{ad}
\DeclareMathOperator*{\grad}{grad}

\begin{document}
\title{On the spectrum of differential operators under Riemannian coverings}
\author{Panagiotis Polymerakis}
\date{}

\maketitle

\renewcommand{\thefootnote}{\fnsymbol{footnote}}
\footnotetext{\emph{Date:} \today} 
\renewcommand{\thefootnote}{\arabic{footnote}}

\renewcommand{\thefootnote}{\fnsymbol{footnote}}
\renewcommand{\thefootnote}{\arabic{footnote}}

\renewcommand{\thefootnote}{\fnsymbol{footnote}}
\renewcommand{\thefootnote}{\arabic{footnote}}

\begin{abstract}
For a Riemannian covering $p \colon M_{2} \to M_{1}$, we compare the spectrum of an essentially self-adjoint differential operator $D_{1}$ on a bundle $E_{1} \to M_{1}$ with the spectrum of its lift $D_{2}$ on $p^{*}E_{1} \to M_{2}$.
We prove that if the covering is infinite sheeted and amenable, then the spectrum of $D_{1}$ is contained in the essential spectrum of any self-adjoint extension of $D_{2}$.
We show that if the deck transformations group of the covering is infinite and $D_{2}$ is essentially self-adjoint (or symmetric and bounded from below), then $D_{2}$ (or the Friedrichs extension of $D_{2}$) does not have eigenvalues of finite multiplicity and in particular, its spectrum is essential.
Moreover, we prove that if $M_{1}$ is closed, then $p$ is amenable if and only if it preserves the bottom of the spectrum of some/any Schr\"{o}dinger operator, extending a result due to Brooks.
\end{abstract}

\section{Introduction}

%OLD PARAGRAPH
%A fundamental problem in Geometric Analysis is to investigate relations between the geometry of a manifold and the spectrum of a differential operator on it. In this direction, it is natural to investigate the behavior of the spectrum under maps which respect the geometry of the manifolds. In this paper, we treat this problem for Riemannian coverings, where the situation is still unclear.

%NEW PARAGRAPH
A basic problem in Geometric Analysis is the investigation of relations between the geometry of a manifold and the spectrum of a differential operator on it. In this direction, it is natural to study the behavior of the spectrum under maps which respect the geometry of the manifolds. In this paper, we deal with this problem for Riemannian coverings.

Let $p \colon M_2 \to M_1$ be a Riemannian covering of connected manifolds with (possibly empty) smooth boundary. A Schr\"{o}dinger operator $S_{1}$ on $M_{1}$ is an operator of the form $S_{1} = \Delta + V$, where $\Delta$ is the (non-negative definite) Laplacian and $V \colon M_{1} \to \mathbb{R}$ is smooth and bounded from below. For such an operator $S_{1}$ on $M_{1}$, its lift on $M_{2}$ is the operator $S_{2} = \Delta + V \circ p$.
%For a Schr\"{o}dinger operator $S_{1}$ on $M_{1}$, its lift $S_{2}$ is the Schr\"{o}dinger operator on $M_{2}$ with potential the lift of the potential of $S_{1}$.
%For a Schr\"{o}dinger operator $S_{1} := \Delta + V$ on $M_{1}$, its lift is the Schr\"{o}dinger operator $S_{2} := \Delta + V \circ p$ on $M_{2}$.
%The first results establishing connections between properties of an infinite sheeted covering and the (Dirichlet) spectrum of a Schr\"{o}dinger operator and its lift are related to the change of the bottom (that is, the minimum) of the spectrum and were proved by Brooks in \cite{MR783536} and \cite{MR656213}.
The first results involving possibly infinite sheeted coverings and establishing connections between properties of the covering and the (Dirichlet) spectra of $S_{1}$ and $S_{2}$, are related to the change of the bottom (that is, the minimum) of the spectrum and were proved by Brooks \cite{MR783536,MR656213}.
He showed that if the underlying manifold is complete, of finite topological type, without boundary and the covering is normal and amenable, then the bottom of the spectrum of the Laplacian is preserved.
%He proved that if $p$ is normal and amenable (that is, the right action of $\pi_{1}(M_{1})$ on $\pi_{1}(M_{2}) \backslash \pi_{1}(M_{1})$ is amenable, which is equivalent to amenability of $\Gamma$, when the covering is normal), $M_{1}$ is complete, of finite topological type and has empty boundary, then $\lambda_{0}(M_1) = \lambda_{0}(M_2)$.
B\'{e}rard and Castillon \cite{MR3104995} extended this result by showing that if the covering is amenable and the underlying manifold is complete with finitely generated fundamental group and without boundary, then the bottom of the spectrum of any Schr\"{o}dinger operator is preserved. Recently, it was proved in \cite{BMP} that the bottom of the spectrum of a Schr\"{o}dinger operator is preserved under amenable coverings, without any topological or geometric assumptions.

In this paper, we prove a global result about this problem in a more general context. Instead of comparing the bottoms of the spectra, we prove inclusion of spectra under some reasonable assumptions. Moreover, our context allows us to impose various boundary conditions on the operators (for instance, Dirichlet, Neumann, mixed and Robin), while the former results involve only Dirichlet conditions. Furthermore, our theorems are applicable to a broad class of differential operators, including Schr\"{o}dinger operators with magnetic potential (that is, first order term), Dirac operators and Schr\"{o}dinger (or Laplace-type) operators on vector bundles. %Actually, our results are partially extended to arbitrary differential operators.
It is worth to point out that the Hodge-Laplacian is a special case of the latter ones.

%In comparison to the already known results, which provide equality of the bottoms of the spectra, Theorem \ref{Inclusion of Spectrums self-adj} provides inclusion of the spectra. Moreover, our results allow us to impose various boundary conditions on the operators (for instance, Dirichlet, Neumann, mixed and Robin), while the former results only involve Dirichlet conditions. Furthermore, our theorems are applicable to many differential operators that are of general interest, such as Schr\"{o}dinger operators with magnetic potential, Schr\"{o}dinger (or Laplace type) operators on vector bundles and Dirac operators. Actually, the general version of Theorem \ref{Inclusion of Spectrums self-adj}, provides a weaker result even for non-symmetric operators.
In order to simplify the statements of our results, we need to set up some notation.
%Let $p\colon M_{2} \to M_{1}$ be a Riemannian covering of connected manifolds with (possibly empty) smooth boundary.
Consider a Riemannian or Hermitian vector bundle $E_{1} \to M_{1}$ endowed with a (not necessarily metric) connection $\nabla$. Let $D_{1}$ be a (not necessarily elliptic) differential operator of arbitrary order on $E_{1}$. We consider the pullback bundle $E_{2} := p^{*}E_{1} \to M_{2}$ endowed with the corresponding metric and connection, and the lift $D_{2}$ of $D_{1}$. 
 
As the domain of $D_{1}$ we consider the space of compactly supported smooth sections $\eta$, which (when $M_{1}$ has non-empty boundary) satisfy a boundary condition of the form $a \nabla_{n} \eta + b \eta = 0$, where $n$ is the inward pointing normal to the boundary and $a$, $b$ are functions on the boundary. The domain of $D_{2}$ is the space of compactly supported smooth sections, which (when the boundary of $M_{1}$ is non-empty) satisfy analogous boundary conditions to the sections in the domain of $D_{1}$. We consider the operators $D_{i}$ restricted to the above domains as densely defined operators in $L^{2}(E_{i})$, $i=1,2$.
 
%Consider a (not necessarily elliptic and of arbitrary order) differential operator $D_{1}$ on a Riemannian or Hermitian vector bundle $E_{1} \to M_{1}$, endowed with a (not necessarily metric) connection $\nabla$. As the domain of this operator we consider the space of compactly supported smooth sections $\varphi$, which (when $M_{1}$ has non-empty boundary) satisfy a boundary condition of the form $\alpha \nabla_{n} \varphi + \beta \varphi = 0$, where $n$ is the inward pointing normal to the boundary and $\alpha$, $\beta$ are functions on the boundary. We consider the pullback bundle $E_{2} \to M_{2}$, endowed with the corresponding metric and connection, and the lift $D_{2}$ of $D_{1}$. The domain of $D_{2}$ is the space of compactly supported smooth sections, which (when the boundary of $M_{1}$ is non-empty) satisfy analogous boundary conditions to the sections in the domain of $D_{1}$. We consider the operators $D_{i}$ restricted to the above domains as densely defined operators in $L^{2}(E_{i})$.

For sake of simplicity, we present here special versions of our main results involving self-adjoint operators. The results are stated for infinite sheeted coverings, since this is the interesting case of amenable coverings. However, we also prove the analogous results for finite sheeted coverings.
Our first result provides inclusion of the spectrum $\sigma(\overline{D}_{1})$ of the closure of $D_{1}$, as long as it is self-adjoint, in the essential spectrum $\sigma_{\ess}(D_{2}^{\prime})$ of any self-adjoint extension $D_{2}^{\prime}$ of $D_{2}$.

%In their general form, they are applicable to arbitrary differential operators. We prove the following:

\begin{theorem}\label{Inclusion of Spectrums self-adj}
%Consider $D_{i}$ restricted to the above domain as a densely defined operator in $L^{2}(E_{i})$, $i=1,2$.
Assume that $D_{1}$ is essentially self-adjoint and let $D_{2}^{\prime}$ be a self-adjoint extension of $D_{2}$. If the covering is infinite sheeted and amenable, then $\sigma(\overline{D}_{1}) \subset \sigma_{\ess}(D_{2}^{\prime})$.
\end{theorem}

Recall that a Schr\"{o}dinger operator on a complete manifold is essentially self-adjoint on the space of compactly supported smooth functions vanishing on the boundary (if it is non-empty). Therefore, in the context of Schr\"{o}dinger operators, it follows that if the underlying manifold is complete and the covering is infinite sheeted and amenable, then the spectrum of $S_{1}$ is contained in the essential spectrum of $S_{2}$.
%the above theorem yields that if $p$ is infinite sheeted and amenable and $M_{1}$ is complete, then $\sigma(S_{1}) \subset \sigma_{ess}(S_{2})$.

An important case where the above theorem cannot be applied is that of Schr\"{o}dinger operators on non-complete Riemannian manifolds.
%On a non-complete manifold, a Schr\"{o}dinger operator restricted to the above domain does not have a unique self-adjoint extension, and we are interested in the spectrum of its Friedrichs extension.
A Schr\"{o}dinger operator on
%a non-complete manifold 
such a manifold does not have a unique self-adjoint extension, when restricted to the above domain, and we are interested in the spectrum of its Friedrichs extension.
%One of the reasons we are interested in non-complete manifolds is that the bottom of the spectrum of a Schr\"{o}dinger operator on a complete manifold with not necessarily smooth boundary, is equal to the bottom of the spectrum of the operator on the interior.
According to \cite{BMP}, if the covering is amenable, then the bottoms of the spectra of $S_{1}$ and $S_{2}$ coincide.
%if $p$ is amenable, then $\lambda_{0}(S_{1}) = \lambda_{0}(S_{2})$, where $\lambda_{0}$ is the bottom of the spectrum.
The amenability is used only to establish $\lambda_{0}(S_{2}) \leq \lambda_{0}(S_{1})$, since the inverse inequality holds for any covering, where $\lambda_{0}$ stands for the bottom of the spectrum. This motivates us to establish the following theorem, which compares the bottom $\lambda_{0}(D_{1}^{(F)})$ of the spectrum of the Friedrichs extension of $D_{1}$ with the bottom $\lambda_{0}^{\ess}(D_{2}^{(F)})$ of the essential spectrum of the Friedrichs extension of $D_{2}$, when the operators are symmetric and bounded from below.

\begin{theorem}\label{Friedrichs}
%Consider $D_{i}$ restricted to the above domain as densely defined operator in $L^{2}(E_{i})$, $i=1,2$.
Assume that $D_{i}$ is symmetric and bounded from below, and denote by $D_{i}^{(F)}$ its Friedrichs extension, $i=1,2$. If the covering is infinite sheeted and amenable, then $\lambda_{0}^{\ess}(D_{2}^{(F)}) \leq \lambda_{0}(D_{1}^{(F)})$.
\end{theorem}

%The proof we present actually establishes the analogous result for any self-adjoint extensions of the operators $D_{i}$, as long as the extension of $D_{1}$ preserves its lower bound.
%Observe that Theorem \ref{Inclusion of Spectrums} yields that if $p$ is infinite sheeted and amenable, $D_{1}$ is essentially self-adjoint and $D_{2}$ is symmetric and bounded from below, then $\sigma(\overline{D}_{1}) \subset \sigma_{ess}(D_{2}^{(F)})$.

%Here, $\lambda_{0}^{ess}$ is the bottom of the essential spectrum of an operator. 
In particular, for Schr\"{o}dinger operators, it follows that if the covering is infinite sheeted and amenable, then the bottom of the spectrum of $S_{1}$ is equal to the bottom of the essential spectrum of $S_{2}$, without any topological or geometric assumptions.
%$\lambda_{0}(S_{1}) = \lambda_{0}^{ess}(S_{2})$.
%Note that we do not have any topological or geometric assumptions.

%In comparison to the already known results, which provide equality of the bottoms of the spectra, Theorem \ref{Inclusion of Spectrums self-adj} provides inclusion of the spectra. Moreover, our results allow us to impose various boundary conditions on the operators (for instance, Dirichlet, Neumann, mixed and Robin), while the former results only involve Dirichlet conditions. Furthermore, our theorems are applicable to many differential operators that are of general interest, such as Schr\"{o}dinger operators with magnetic potential, Schr\"{o}dinger (or Laplace type) operators on vector bundles and Dirac operators. Actually, the general version of Theorem \ref{Inclusion of Spectrums self-adj}, provides a weaker result even for non-symmetric operators.

%Our results are stated for infinite sheeted coverings. This is the interesting case of amenable coverings, since finite sheeted coverings are amenable, but significantly simpler. For sake of completeness, we also present a couple of lemmas with analogous results.

The above results involve amenable coverings. However, the deck transformations group of a (possibly non-amenable) covering provides information about the group of isometries of the covering space.
This motivates us to work in a more general context than Riemannian coverings and prove that under some symmetry assumptions, an essentially self-adjoint differential operator does not have eigenvalues of finite multiplicity and in particular, its
spectrum is essential. Moreover, we show the analogous result for the Friedrichs extension of a symmetric and bounded from below differential operator.
%approximate point spectrum of the closure of the operator is equal to the Weyl spectrum and the operator does not have eigenvalues of finite multiplicity. Moreover, we show the analogous result for the Friedrichs extension of a symmetric and bounded from below differential operator.
In the context of Riemannian coverings, we obtain the following immediate consequences.

\begin{corollary}\label{Infinite Deck Transformations self-adj}
Assume that $D_{2}$ is essentially self-adjoint. If the deck transformations group of the covering is infinite, then $\overline{D}_{2}$ does not have eigenvalues of finite multiplicity and in particular, $\sigma(\overline{D}_{2}) = \sigma_{\ess}(\overline{D}_{2})$.
\end{corollary}

\begin{corollary}\label{Infinite Deck Transformations Friedrichs}
Assume that $D_{2}$ is symmetric and bounded from below, and denote by $D_{2}^{(F)}$ its Friedrichs extension. If the deck transformations group of the covering is infinite, then $D_{2}^{(F)}$ does not have eigenvalues of finite multiplicity and $\sigma(D_{2}^{(F)}) = \sigma_{\ess}(D_{2}^{(F)})$.
\end{corollary}

For Schr\"{o}dinger operators, it follows that if the deck transformations group of the covering is infinite, then the spectrum of $S_{2}$ is essential, without any assumptions on the manifolds.
%It is worth noticing that the above corollaries do not apply to arbitrary infinite sheeted, amenable coverings (see Example \ref{infinite sheeted amenable trivial dtg}).

All the above results provide information about the spectra from properties of the covering (amenability or infinite deck transformations group).
In the converse direction, Brooks \cite{MR656213} proved that
if a normal Riemannian covering of a closed manifold (that is, compact without boundary) preserves the bottom of the spectrum of the Laplacian, then the covering is amenable.
%if $p$ is normal, $M_{1}$ is closed (that is, compact without boundary) and $\lambda_{0}(M_{2}) = 0$, then $p$ is amenable. 
%Note that this result is trivial for finite sheeted coverings. 
%Brooks' proof of this theorem is quite complicated and relies heavily on geometric measure theory.
In this paper, we extend this result to Schr\"{o}dinger operators and to not necessarily normal coverings. In the following theorem, we denote by $h^{\ess}(M)$ the supremum of the Cheeger's constants over complements of compact and smoothly bounded domains of $M$. 

\begin{theorem}\label{Improved Brooks}
Let $p \colon M_{2} \to M_{1}$ be a Riemannian covering with $M_{1}$ closed. %Let $S_{1}$ be a Schr\"{o}dinger operator on $M_{1}$ and $S_{2}$ its lift on $M_{2}$. 
Then the following are equivalent:
\begin{enumerate}[topsep=0pt,itemsep=-1pt,partopsep=1ex,parsep=0.5ex,leftmargin=*, label=(\roman*), align=left, labelsep=0em]
\item $p$ is infinite sheeted and amenable,

\item $\sigma(S_{1}) \subset \sigma_{\ess}(S_{2})$ for some/any Schr\"{o}dinger operator $S_{1}$ on $M_{1}$ and its lift $S_{2}$,

\item $\lambda_{0}(S_{1}) = \lambda_{0}^{\ess}(S_{2})$ for some/any Schr\"{o}dinger operator $S_{1}$ on $M_{1}$ and its lift $S_{2}$,

\item $h^{\ess}(M_{2}) = 0$.

\end{enumerate}
\end{theorem}

It is worth to point out that Brooks proved his theorem in a quite complicated way, relying heavily on geometric measure theory. Our proof of the above theorem is significantly simpler and avoids the use of geometric measure theory.
%In order to establish Theorem \ref{Improved Brooks}, we provide an elementary proof for it, extending it to not necessarily normal coverings.

%It is worth to point out that in particular, the above theorem yields that when the underlying manifold is closed, then any Schr\"{o}dinger operator realizes whether the covering is amenable or not.

Furthermore, Brooks \cite{MR783536} proved that under some more general (but still quite restrictive) assumptions, if the bottom of the spectrum of the Laplacian is preserved, then the covering is amenable. In particular, these assumptions imply that the bottom of the spectrum of the Laplacian on $M_{1}$ is not in the essential spectrum.
%$\lambda_{0}(M_{1}) \notin \sigma_{ess}(M_{1})$.
Moreover, he provided examples demonstrating that without these conditions, the bottom of the spectrum of the Laplacian may be preserved even if the covering is non-amenable. This suggests that under some assumptions on the geometry and the spectrum of the Laplacian on $M_{1}$, the bottom of the spectrum is preserved under a weaker assumption than amenability of the covering. In this direction we prove the following result.

\begin{corollary}\label{Weaker assumption}
Let $p \colon M_{2} \to M_{1}$ be a Riemannian covering with $M_{1}$ complete. Let $S_{1}$ be a Schr\"{o}dinger operator on $M_{1}$ with $\lambda_{0}(S_{1}) \in \sigma_{\ess}(S_{1})$, and $S_{2}$ its lift on $M_{2}$. If there exists a compact $K \subset M_{1}$, such that the image of the fundamental group of any connected component of $M_{1} \smallsetminus K$ in $\pi_{1}(M_{1})$ is amenable, then $\lambda_{0}(S_{1}) =\lambda_{0}(S_{2})$.
\end{corollary}

The paper is organized as follows: In Section \ref{Preliminaries}, we give some preliminaries.
%review some well-known facts that are used throughout the paper.
In Sections \ref{extension} and \ref{amenable section}, we present the construction which is used in order to prove Theorem \ref{Friedrichs} and a more general result (Theorem \ref{Inclusion of Spectrums}) than Theorem \ref{Inclusion of Spectrums self-adj}. The proofs are given in Section \ref{amenable section}, where we also present the analogous results for finite sheeted coverings. In Section \ref{high symmetry section}, we study manifolds with high symmetry and establish Corollaries \ref{Infinite Deck Transformations self-adj} and \ref{Infinite Deck Transformations Friedrichs}. In Section \ref{closed underlying section}, we present an alternative proof of Brooks' theorem \cite{MR656213}, extending it to not necessarily normal Riemannian coverings. In Section \ref{applications section}, we introduce the notion of renormalized Schr\"{o}dinger operators, which is used to prove Theorem \ref{Improved Brooks}. Moreover, in this section we establish Corollary \ref{Weaker assumption} and we present a simple example demonstrating that the behavior of the bottom of the spectrum of the connection Laplacian under a covering depends 
%that whether or not the bottom of the spectrum of the connection Laplacian is preserved under a Riemannian covering depends 
on the corresponding metric connection.
Therefore, a main point in our results is the independence from the vector bundles, the connections and the differential operators. \medskip

\textbf{Acknowledgements.} I would like to thank Werner Ballmann and Dorothee Sch\"{u}th for some very enlightening discussions and helpful remarks. I am also grateful to the Max Planck Institute for Mathematics in Bonn for its support and hospitality.

\section{Preliminaries}\label{Preliminaries}

We first recall some basic facts from functional analysis. For more details, see \cite{MR1361167}. Let $A \colon \mathcal{D}(A) \subset \mathcal{H} \to \mathcal{H} $ be a closed (linear) operator on a separable Hilbert space $\mathcal{H}$ over a field $\mathbb{F}$, where $\mathbb{F} = \mathbb{R}$ or $\mathbb{F} = \mathbb{C}$. The \emph{spectrum} of $A$ is given by
\[
\sigma(A) := \{ \lambda \in \mathbb{F} :(A - \lambda) \colon \mathcal{D}(A) \to \mathcal{H} \text{ not bijective} \}.
\]
The \emph{essential spectrum} of $A$ is defined as 
\[
\sigma_{\ess}(A) := \{ \lambda \in \mathbb{F} :(A - \lambda) \colon \mathcal{D}(A) \to \mathcal{H} \text{ not Fredholm} \}.
\]
Recall that an operator is called \emph{Fredholm} if its kernel is finite dimensional and its range is closed and of finite codimension. 
%In general, we think of the essential spectrum as the singular part of the spectrum of an operator. 
The \emph{discrete spectrum} of $A$ is the complement of the essential spectrum in the spectrum of $A$, that is, $\sigma_{d} (A) := \sigma(A) \smallsetminus \sigma_{\ess}(A)$.

The \emph{approximate point spectrum} of $A$, denoted by $\sigma_{\ap}(A)$, is defined as the set of all $\lambda \in \mathbb{F} $, such that there exists $(v_{k})_{k \in \mathbb{N}} \subset \mathcal{D}(A)$ with $\| v_{k} \| = 1$ and $(A - \lambda)v_{k} \rightarrow 0$ in $\mathcal{H}$. For $\lambda \in \mathbb{F}$, a \emph{Weyl sequence} for $A$ and $\lambda$ is a sequence $(v_{k})_{k \in \mathbb{N}} \subset \mathcal{D}(A)$, such that $\| v_{k} \| = 1$, $v_{k} \rightharpoonup 0$ and $(A - \lambda)v_{k} \rightarrow 0$ in $\mathcal{H}$, where ``$\rightharpoonup$" denotes the weak convergence in $\mathcal{H}$. The \emph{Weyl spectrum} of $A$, denoted by $\sigma_{W}(A)$, is the set of all $\lambda \in \mathbb{F}$, such that there exists a Weyl sequence for $A$ and $\lambda$.

The following proposition is the characterization of the spectrum of a self-adjoint operator as the set of approximate eigenvalues and the well-known Weyl's criterion for the essential spectrum.

\begin{proposition}\label{Spectrum of Self-adj}
If $A$ is self-adjoint, then $\sigma_{\ap}(A) = \sigma(A)$, $\sigma_{W}(A) = \sigma_{\ess}(A)$ and $\sigma_{d}(A)$ consists of isolated eigenvalues of $A$ of finite multiplicity.
\end{proposition}

Since we are interested in closures of operators, we need the following elementary lemma, characterizing the approximate point spectrum and the Weyl spectrum of the closure in terms of the initial operator.

\begin{lemma}\label{Spectrum of Closure}
Assume that $A$ is the closure of an operator $B \colon \mathcal{D}(B) \subset \mathcal{H} \to \mathcal{H}$ and consider $\lambda \in \mathbb{F}$. Then:
\begin{enumerate}[topsep=0pt,itemsep=-1pt,partopsep=1ex,parsep=0.5ex,leftmargin=*, label=(\roman*), align=left, labelsep=0em]
\item $\lambda \in \sigma_{\ap}(A)$ if and only if there exists $(v_{k})_{k \in \mathbb{N}} \subset \mathcal{D}(B)$, such that $\| v_{k} \| = 1$ and $(B - \lambda)v_{k} \rightarrow 0$ in $\mathcal{H}$,
\item $\lambda \in \sigma_{W}(A)$ if and only if there exists $(v_{k})_{k \in \mathbb{N}} \subset \mathcal{D}(B)$, such that $\| v_{k} \| = 1$, $v_{k} \rightharpoonup 0$ and $(B - \lambda)v_{k} \rightarrow 0$ in $\mathcal{H}$.
\end{enumerate}
\end{lemma}

For an operator $B \colon \mathcal{D}(B) \subset \mathcal{H} \to \mathcal{H}$ and $v \in \mathcal{D}(B) \smallsetminus \{0\}$, the \textit{Rayleigh quotient} of $v$ with respect to $B$ is defined as
\[
\mathcal{R}_{B}(v) := \frac{\langle Bv,v \rangle}{\| v \|^{2}}.
\]
If $B$ is symmetric, then $\mathcal{R}_{B}(v) \in \mathbb{R}$, for any $v \in \mathcal{D}(B) \smallsetminus \{0\}$, and $B$ is \textit{bounded from below} if the infimum of $\mathcal{R}_{B}(v)$, with $v \in \mathcal{D}(B) \smallsetminus \{0\}$, is finite. In this case, this infimum is called the \textit{lower bound} of $B$.

The spectrum of a self-adjoint operator $A$ is contained in $\mathbb{R}$ and the \textit{bottom} (that is, the minimum) of the spectrum and the bottom of the essential spectrum of $A$ are denoted by $\lambda_{0}(A)$ and $\lambda_{0}^{\ess}(A)$, respectively. The following characterization of the bottom of the spectrum is due to Rayleigh.

\begin{proposition}\label{Rayleigh}
If $A \colon \mathcal{D}(A) \subset \mathcal{H} \to \mathcal{H}$ is self-adjoint, then
\[
\lambda_{0}(A) = \inf_{v \in \mathcal{D}(A) \smallsetminus \{0\}} \mathcal{R}_{A}(v).
\]
If, in addition, $A$ is the closure of an operator $B \colon \mathcal{D}(B) \subset \mathcal{H} \to \mathcal{H}$, then the bottom of the spectrum of $A$ is given by
\[
\lambda_{0}(A) = \inf_{v \in \mathcal{D}(B) \smallsetminus \{0\}} \mathcal{R}_{B}(v).
\]
\end{proposition}

Throughout the paper, manifolds are connected, with possibly empty, smooth and not necessarily connected boundary, unless otherwise stated.
Let $p \colon M_{2} \to M_{1}$ be a Riemannian covering of $m$-dimensional manifolds, $E_{1} \to M_{1}$ a Riemannian or Hermitian vector bundle of rank $\ell$ and $D_{1} \colon \Gamma (E_{1}) \to \Gamma(E_{1})$ a differential operator of order $d$. Consider the pullback bundle $E_{2} := p^{*}E_{1}$ on $M_{2}$, $y \in M_{2}$ and set $x := p(y)$. Let $U_{2}$ be an open neighborhood of $y$, such that the restriction $p|_{U_{2}}$ is an isometry onto its image $U_{1}$. The \emph{lift} $D_{2}  \colon \Gamma(E_{2}) \to \Gamma(E_{2})$ of $D_{1}$ is the differential operator defined by
\[
D_{2}\eta(z) := (p|_{U_{2}})^{*}(D_{1}((p|_{U_{2}}^{-1})^{*}\eta)(p(z))),
\]
for any $\eta \in \Gamma(E_{2})$ and $z \in U_{2}$.
Without loss of generality, we may assume that $U_{1}$ is contained in a coordinate neighborhood and there exists a trivialization $E_{1}|_{U_{1}} \to U_{1} \times \mathbb{F}^{\ell}$.
%Without loss of generality, we may assume that $U_{1}$ is contained in a coordinate neighborhood and there exists a trivialization $E_{1}|_{U_{1}} \to U_{1}^{\prime} \times \mathbb
%{F}^{\ell}$, with $U_{1}^{\prime} \subset \mathbb{H}^{m}$ open, where $\mathbb{H}^{m}$ is the $m$-dimensional Euclidean half space (since $M_{1}$ may have non-empty boundary). 
With respect to this coordinate system and trivialization, $D_{1}$ is expressed as
\begin{equation}\label{Local Expression}
D_{1} = \sum_{|\alpha| \leq d} A^{\alpha} \frac{\partial^{|\alpha|}}{\partial x^{\alpha}},
\end{equation}
where $A^{\alpha}$ are smooth maps defined on $U_{1}$, with values $\ell \times \ell$ matrices with entries in $\mathbb{F}$.
Then, with respect to the lifted coordinate system and the corresponding trivialization $E_{2}|_{U_{2}} \to U_{2} \times \mathbb{F}^{\ell}$, $D_{2}$ has the local expression
\[
D_{2} = \sum_{|\alpha| \leq d} (A^{\alpha} \circ p) \frac{\partial^{|\alpha|}}{\partial y^{\alpha}}.
\]
%where $\partial / \partial y_{i} := p^{*}(\partial / \partial x_{i})$, for $i=1,\dots,m$.
\begin{lemma}\label{closability}
Let $M$ be a Riemannian manifold, $E \to M$ a Riemannian or Hermitian vector bundle endowed with a connection $\nabla$ and $D \colon \Gamma(E) \to \Gamma(E)$ a differential operator. If $M$ has empty boundary, set $\mathcal{D}(D) := \Gamma_{c}(E)$. If $M$ has non-empty boundary, let $a, b$ be real or complex valued functions (depending on whether $E$ is Riemannian or Hermitian) defined on $\partial M$, let $n$ be the inward pointing normal to $\partial M$ and consider
\[
\mathcal{D}(D) := \{ \eta \in \Gamma_{c}(E) : a \nabla_{n} \eta + b \eta = 0 \text{ on } \partial M \}.
\]
Then the operator $D \colon \mathcal{D}(D) \subset L^{2}(E) \to L^{2}(E)$ is closable.
\end{lemma}

\begin{proof}
Consider the formal adjoint $D^{\ad}$ of $D$, defined by
\[
\langle D \eta , \theta \rangle = \langle \eta , D^{\ad} \theta \rangle,
\]
for all $\eta \in \mathcal{D}(D)$ and $\theta \in \Gamma_{cc}(E)$, where $\Gamma_{cc}(E)$ is the space of smooth sections, compactly supported in the interior of $M$. It is clear that the operator $D^{\ad} \colon \Gamma_{cc}(E) \subset L^{2}(E) \to L^{2}(E)$ is densely defined and its adjoint satisfies $D \subset (D^{\ad})^{*}$. Since the adjoint is closed, it follows that $D$ is closable. \qed
\end{proof}
\medskip

%Recall that a Schr\"{o}dinger operator on a connected Riemannian manifold $M$ (with possibly empty, smooth boundary) is an 
%Let $M$ be a connected Riemannian manifold with (possibly empty) smooth boundary. A Schr\"{o}dinger operator on $M$ is an 
%operator of the form $S := \Delta + V$, where $\Delta$ is the (non-negative definite) Laplacian and the potential $V \colon M \to \mathbb{R}$ is smooth and bounded from below. Recall that if $M$ is complete, then $S$ is essentially self-adjoint on the space of compactly supported smooth functions vanishing on the boundary, if $M$ has non-empty boundary. If $M$ is non-complete, then $S$ restricted on this domain, does not have a unique self-adjoint extension and we are interested in the spectrum of its Friedrichs extension.

A \textit{Schr\"{o}dinger operator} on a possibly non-connected Riemannian manifold $M$ is an operator of the form $S := \Delta + V$, where $\Delta$ is the Laplacian and $V \colon M \to \mathbb{R}$ is smooth and bounded from below. 
If $M$ is complete and without boundary, then $S$ is essentially self-adjoint on $C^{\infty}_{c}(M)$, that is, the closure of $S \colon C^{\infty}_{c}(M) \subset L^{2}(M) \to L^{2}(M)$ is self-adjoint. If $M$ is complete with non-empty boundary, then $S$ is essentially self-adjoint on $\{ f \in C^{\infty}_{c}(M) : f = 0 \text{ on } \partial M \}$. If $M$ is non-complete, then $S$ restricted to the above domain, does not have a unique self-adjoint extension, and we are interested in the Friedrichs extension of $S$. By abuse of notation, the spectrum and the essential spectrum of the above described self-adjoint operator are denoted by $\sigma(S)$ and $\sigma_{\ess}(S)$, respectively, and their bottoms by $\lambda_{0}(S)$ and $\lambda_{0}^{\ess}(S)$, respectively. These sets and quantities for the Laplacian on $M$ are denoted by $\sigma(M)$, $\sigma_{\ess}(M)$ and $\lambda_{0}(M)$, $\lambda_{0}^{\ess}(M)$, respectively.

Let $p \colon M_{2} \to M_{1}$ be a Riemannian covering of complete manifolds without boundary. For $x \in M_{1}$ and $y \in p^{-1}(x)$, the \emph{fundamental domain} of $p$ centered at $y$ is defined by
\[
D_{y} := \{ z \in M_{2} : d(z,y) \leq d(z,y^{\prime}) \text{ for all } y^{\prime} \in p^{-1}(x) \}.
\]
Some basic properties of these fundamental domains are presented in \cite{BMP}. It is clear that $D_{y}$ is closed and $M_{2}$ is the union of $D_{y}$, with $y \in p^{-1}(x)$. Moreover, $\partial D_{y}$ and the cut locus $\Cut (x)$ of $x$ are of measure zero and $p \colon D_{y} \smallsetminus \partial D_{y} \to M_{1} \smallsetminus C_{0}$ is an isometry, where $C_{0}$ is a subset of $\Cut (x)$.
The following two lemmas are proved in \cite{BMP}. The lemma after these is proved similarly to Lemma \ref{Intersection Lemma}. In these lemmas and in the sequel, we denote open and closed balls by $B$ and $C$, respectively.
\begin{lemma}\label{preimage of compact}
If $K \subset B(x,r)$, then $p^{-1}(K) \cap D_{y} \subset B(y,r)$. In particular, if $K$ is compact, then $p^{-1}(K) \cap D_{y}$ is compact.
\end{lemma}

\begin{lemma}\label{Intersection Lemma}
For any $r >0$, there exists $N(r) \in \mathbb{N}$, such that any $z \in M_{2}$ is contained in at most $N(r)$ of the balls $C(y,r)$, with $y \in p^{-1}(x)$.
\end{lemma}

\begin{lemma}\label{Cardinality Estimate Lemma}
Consider the universal coverings $p_{i} \colon \tilde{M} \to M_{i}$, $i=1,2$.
%Consider the simply connected covering space $\tilde{M}$ of $M_{i}$ and the coverings $p_{i} \colon  \tilde{M} \to M_{i}$.
For any $r$, $r_{0} > 0$, there exists $\tilde{N}(r,r_{0}) \in \mathbb{N}$, such that $$\#\{ w \in p_{2}^{-1}(z) : B(w,r_{0}) \cap C(u,r) \neq \emptyset \} \leq \tilde{N}(r,r_{0}),$$ for all $u \in p_{1}^{-1}(x)$ and $z \in M_{2}$.
\end{lemma}

Finally, we recall the notions of amenable right action and amenable covering. For more details on amenable left actions, which are completely analogous to right actions, see \cite[Section 2]{MR3104995}. A right action of a countable group $\Gamma$ on a countable set $X$ is called \emph{amenable} if there exists a $\Gamma$-invariant mean on $L^{\infty}(X)$. The following characterization is due to F\o{}lner.

\begin{proposition}\label{Folner}
The right action of a countable group $\Gamma$ on a non-empty, countable set $X$ is amenable if and only if for any finite $G \subset \Gamma$ and $\varepsilon > 0$, there exists a non-empty, finite $F \subset X$, such that
\[
\#(F \smallsetminus Fg) < \varepsilon \#(F),
\]
for all $g \in G$. Such a set $F$ is called a \emph{F\o{}lner set} for $G$ and $\varepsilon$.
\end{proposition}
A countable group $\Gamma$ is called \emph{amenable} if the right action of $\Gamma$ on itself is amenable. In this case, the right action of $\Gamma$ on any countable set $X$ is amenable. Moreover, it is clear that any right action on a non-empty, finite set is amenable.
%note that for a finite $X$, any right action on $X$ is amenable.

A Riemannian covering $p \colon M_2 \to M_1$ is called \textit{amenable} if the right action of $\pi_{1}(M_1)$ on $\pi_{1}(M_{2}) \backslash \pi_{1}(M_{1})$ (that is, the set of right cosets of $\pi_{1}(M_{2})$ in $\pi_{1}(M_{1})$, when considered as deck transformations groups of the universal coverings) is amenable. Clearly, a normal covering is amenable if and only if its deck transformations group is amenable.
%It is clear that if $p$ is a normal covering with deck transformations group $\Gamma$, then $p$ is amenable if and only if $\Gamma$ is amenable.
Furthermore, finite sheeted coverings are amenable.

The following criteria for amenability of groups are immediate consequences of the definition and Proposition \ref{Folner}.

\begin{corollary}\label{subexp growth}
Any finitely generated group of subexponential growth is amenable.
%Let $\Gamma$ be a finitely generated group. If $\Gamma$ has subexponential growth, then $\Gamma$ is amenable.
\end{corollary}

\begin{corollary}\label{finitely generated subgroups}
A countable group $\Gamma$ is amenable if and only if any finitely generated subgroup of $\Gamma$ is amenable
%Let $\Gamma$ be a countable group. Then $\Gamma$ is amenable if and only if every finitely generated subgroup of $\Gamma$ is amenable.
\end{corollary}

\begin{corollary}\label{solvable}
Any countable solvable group is amenable.
%If $\Gamma$ is a countable, solvable group, then $\Gamma$ is amenable.
\end{corollary}

\begin{proof}
From Corollaries \ref{subexp growth} and \ref{finitely generated subgroups}, it follows that any countable abelian group is amenable. From the definition, it is clear that an extension of an amenable group by an amenable group is also amenable.\qed
\end{proof}

\section{Coverings of manifolds with boundary}\label{extension}

The aim of this section is to show the following proposition, according to which, any Riemannian covering of manifolds with boundary can be ``extended" to a Riemannian covering of manifolds without boundary.

\begin{proposition}\label{Extension of Covering}
Let $M$ be a Riemannian manifold with non-empty boundary. Then there exists a Riemannian manifold $N$ of the same dimension, without boundary and an isometric embedding $i \colon M \to N$, such that, after identifying $M$ with $i(M)$, any Riemannian covering $p \colon M^{\prime} \to M$ can be extended to a Riemannian covering $p \colon N^{\prime} \to N$.
\end{proposition}

In order to prove this proposition, we need to establish some auxiliary lemmas.

\begin{lemma}
Let $M$ be a Riemannian manifold with non-empty boundary. Then there exists a Riemannian manifold $N$ of the same dimension, without boundary, an isometric embedding $i \colon M \to N$ and a strong deformation retraction of $N$ onto $i(M)$.
\end{lemma}

\begin{proof}
Consider the space $\partial M \times [0, +\infty )$ and the map $\Psi \colon \partial M \to \partial M \times [0, +\infty )$, defined by $\Psi(x) := (x,0)$. Then $N := M \cup_{\Psi} (\partial M \times [0,+ \infty))$ is a smooth manifold and there exists a smooth embedding $i \colon M \to N$. Therefore, $M$ can be identified with $i(M)$. Since $M$ is connected, so is $N$, and there exists a strong deformation retraction of $N$ onto $M$, obtained by considering $F_{t}(x,r) := (x,(1-t)r)$ in the glued ends $\partial M \times [0,+\infty)$.

It remains to extend the Riemannian metric of $M$ to a Riemannian metric of $N$. Any $x \in \partial M$ has an open neighborhood $U_{x}$ in $N$, such that there exists a smooth frame field $\{e_{1}, \dots , e_{m}\}$ in $U_{x}$, where $m$ is the dimension of the manifolds. Let $g_{jk} := \langle e_{j}, e_{k} \rangle$, $1 \leq j,k \leq m$, be the components of the Riemannian metric of $M$. Since they are smooth up to the boundary of $M$, they can be extended smoothly to a neighborhood of $x$. After passing to a smaller neighborhood of $x$ if needed, we may assume that $g_{jk}$'s are smooth in $U_{x}$ and their matrix is symmetric and positive definite at any point of $U_{x}$. Hence, they express a Riemannian metric in $U_{x}$.

Clearly, $\partial M$ can be covered with such neighborhoods $U_{x}$. Consider the interior of $M$ as an open subset of $N$ endowed with its Riemannian metric and $N \smallsetminus M$ with an arbitrary Riemannian metric. Combining these Riemannian metrics via a partition of unity subordinate to this open cover of $N$, gives rise to a Riemannian metric of $N$, which is an extension of the Riemannian metric of $M$. \qed
\end{proof}

\begin{lemma}
Let $M$ be a Riemannian manifold with non-empty boundary. Consider $N$ as in the previous lemma and identify $M$ with $i(M)$. Let $q \colon \tilde{N} \to N$ be the universal covering of $N$. Then the restriction $q \colon q^{-1}(M) \to M$ is the universal covering of $M$.
\end{lemma}

\begin{proof}
Since there exists a strong deformation retraction of $N$ onto $M$, every loop in $N$ can be homotoped to a loop in $M$. This implies that for any $x \in M$ and $y_{1}, y_{2} \in q^{-1}(x)$, there exists a path in $q^{-1}(M)$ from $y_{1}$ to $y_{2}$. Since $M$ is connected, it follows that so is $q^{-1}(M)$ and the restriction $ q \colon q^{-1}(M) \to M$ is a covering of (connected) manifolds.

Let $r_{M} \colon N \to M$ be a retraction. Then the map $r_{M} \circ q \colon \tilde{N} \to M$ is continuous and $r_{M} \circ q = q$ in $q^{-1}(M)$. From the Lifting Theorem, it has a continuous lift $\tilde{r}_{M} \colon \tilde{N} \to q^{-1}(M)$, with $\tilde{r}_{M}(y_{0}) = y_{0}$, for some $y_{0} \in q^{-1}(M)$. Since $\tilde{r}_{M}|_{q^{-1}(M)}$ is a deck transformation of the covering $q \colon q^{-1}(M) \to M$, it follows that $\tilde{r}_{M} \colon \tilde{N} \to q^{-1}(M)$ is a retraction. Since $\tilde{N}$ is simply connected, this yields that so is $q^{-1}(M)$.
%it can be lifted to a retraction $\tilde{r} \colon \tilde{N} \to q^{-1}(M)$. Since $\tilde{N}$ is simply connected, this yields that so is $q^{-1}(M)$.
%it has a continuous lift $\tilde{r} \colon \tilde{N} \rightarrow q^{-1}(M) $, which is the identity in $q^{-1}(M)$. Denote by $j$ the inclusion of $q^{-1}(M)$ to $\tilde{N}$. Then $\tilde{r} \circ j = \text{Id}_{q^{-1}(M)}$, which yields that for the corresponding maps between the fundamental groups, we have $\tilde{r}_{*} \circ j_{*} = \text{Id}_{\pi_{1}(q^{-1}(M))}$. In particular, $\tilde{r}_{*} \colon \pi_{1}(\tilde{N}) \to \pi_{1}(q^{-1}(M))$ is surjective and since $\tilde{N}$ is simply connected, it follows that so is $q^{-1}(M)$.
\qed
\end{proof}
\medskip

\noindent{\emph{Proof of Proposition \ref{Extension of Covering}:}} Consider $N$ and $q \colon \tilde{N} \to N$ as in the above lemmas, identify $M$ with $i(M)$ and set $\tilde{M} := q^{-1}(M)$. Denote by $\Gamma_{N}$ and $\Gamma_{M}$ the deck transformations groups of $q \colon \tilde{N} \to N$ and $q \colon \tilde{M} \to M$, respectively. It is clear that for $g \in \Gamma_{N}$, we have $g|_{\tilde{M}} \in \Gamma_{M}$, and any $\gamma \in \Gamma_{M}$ has a unique extension $\gamma^{\prime} \in \Gamma_{N}$.
For any Riemannian covering $p \colon M^{\prime} \to M$, there exists a subgroup $\Gamma \subset \Gamma_{M}$, such that $M^{\prime} = \tilde{M} / \Gamma$. For $\Gamma^{\prime} := \{ \gamma^{\prime} \in \Gamma_{N} : \gamma \in \Gamma \}$ and $N^{\prime} := \tilde{N} / \Gamma^{\prime}$, the inclusion $\tilde{M} \hookrightarrow \tilde{N}$ descends to an isometric embedding $M^{\prime} \to N^{\prime}$, which completes the proof. \qed

\section{Spectrum of operators under amenable coverings}\label{amenable section}

%In this section, we study the behavior of the spectrum under amenable coverings. 
Throughout this section, we work in the following context, which is briefly described in the Introduction.

Let $p \colon M_{2} \to M_{1}$ be a Riemannian covering, $E_{1} \to M_{1}$ a Riemannian or Hermitian vector bundle endowed with a connection $\nabla$ and $D_{1} \colon \Gamma(E_{1}) \to \Gamma(E_{1})$ a differential operator on $E_{1}$. Let $E_{2} \to M_{2}$ be the pullback bundle, endowed with the corresponding metric and connection $\nabla$, and $D_{2} \colon \Gamma(E_{2}) \to \Gamma(E_{2})$ the lift of $D_{1}$.
If $M_{1}$ has empty boundary, we consider the space of compactly supported smooth sections of $E_i$ as the domain of $D_{i}$, that is, $\mathcal{D}(D_{i}) := \Gamma_{c}(E_{i})$, $i=1,2$. If $M_{1}$ has non-empty boundary, the domain of $D_{i}$ is the space 
$$\mathcal{D}(D_{i}) := \{ \eta \in \Gamma_{c}(E_{i}) : a_{i} \nabla_{n_{i}} \eta + b_{i} \eta = 0 \text{ on } \partial M_{i} \},$$
where $n_{i}$ is the inward pointing normal to $\partial M_{i}$, $i=1,2$, $a_{1}$, $b_{1}$ are real or complex valued functions (depending on whether the bundles are Riemannian or Hermitian) on $\partial M_{1}$, and $a_{2} = a_{1} \circ p$, $b_{2} = b_{1} \circ p$. It is worth to point out that we do not impose any assumptions on $a_{1}$ and $b_{1}$.
%Note that we do not impose any assumptions on $\alpha_{i}$ and $\beta_{i}$.
When we refer to closability, symmetry or essential self-adjointness of $D_{i}$, we consider the operator $D_{i} \colon \mathcal{D}(D_{i}) \subset L^{2}(E_{i}) \to L^{2}(E_{i})$, $i=1,2$. From Lemma \ref{closability}, the operator $D_{i}$ is closable and we denote by $\overline{D}_{i}$ its closure, $i=1,2$.

Our aim in this section is prove Theorem \ref{Friedrichs} and the following more general version of Theorem \ref{Inclusion of Spectrums self-adj}.

\begin{theorem}\label{Inclusion of Spectrums}
Let $D_{2}^{\prime}$ be a closed extension of $D_{2}$. If the covering is infinite sheeted and amenable, then $\sigma_{\ap}(\overline{D}_{1}) \subset \sigma_{W}(D_{2}^{\prime})$.
\end{theorem}

\subsection{Partition of unity}

In this subsection, we construct a special partition of unity, which is used in the sequel to obtain cut-off functions on $M_{2}$.
%Throughout this section, as well as in the next one, we work in the general context described in the Introduction.

Consider the universal coverings
%simply connected covering space $\tilde{M}$ of $M_{i}$, the coverings 
$p_{i} \colon \tilde{M} \to M_{i}$ and denote by $\Gamma_{i}$ the deck transformations group of $p_i$, $i=1,2$.
If $M_{1}$ has empty boundary, consider a Riemannian metric $\mathfrak{h}$, conformal to the original metric $\mathfrak{g}$, such that $(M_{1},\mathfrak{h})$ is complete. If $M_{1}$ has non-empty boundary, consider a Riemannian manifold $N_{1}$ containing $M_{1}$, as in Proposition \ref{Extension of Covering}, and a Riemannian metric $\mathfrak{h}$, conformal to the original metric $\mathfrak{g}$, such that $(N_{1},\mathfrak{h})$ is complete. From now on, geodesics are considered with respect to $\mathfrak{h}$ and its lifts.
We denote by $\grad f$ and $\grad_{\mathfrak{h}}f$ the gradient of a function $f$ with respect to $\mathfrak{g}$ and $\mathfrak{h}$ (or their lifts), respectively.
If $M_{1}$ has empty boundary, distances are considered with respect to $\mathfrak{h}$ or its lifts. In this case, we denote the open (respectively, closed) ball of radius $r$ around a point $z$ by $B(z,r)$ (respectively, $C(z,r)$).
If $M_{1}$ has non-empty boundary, the distance between two points is considered in $(N_{1},\mathfrak{h})$ or its corresponding covering space. In this case, $B(z,r)$ and $C(z,r)$ stand for the corresponding balls in $M_{1}$, $M_{2}$ or $\tilde{M}$. %with respect to the distance function $d(\cdot,\cdot)$ of $N_{1}$ or its corresponding covering space.

Consider $x \in M_{1}$, $u \in p_{1}^{-1}(x)$ and $r>0$ large enough, so that $B(u,r) \cap \partial \tilde{M} \neq \emptyset$, if $M_{1}$ has non-empty boundary.

\begin{lemma}
There exists a non-negative $\psi_{u} \in C^{\infty}_{c}(\tilde{M})$, such that $\supp \psi_{u} \subset C(u,r+1)$ and $\psi_{u} =1$ in $C(u,r+ 1/2)$. Moreover, if $M_{1}$ has non-empty boundary, $\psi_{u}$ can be chosen such that $\grad \psi_{u}$ is tangential to $\partial \tilde{M}$.
\end{lemma}

\begin{proof}
It is clear that there exists a non-negative $\psi_{u}^{\prime} \in C^{\infty}_{c}(\tilde{M})$ with $\supp \psi_{u}^{\prime} \subset C(u,r+1)$ and $\psi_{u}^{\prime} = 1$ in $C(u,r+ 1/2)$. If $M_{1}$ has empty boundary, this is the desired function. Otherwise, let $K := \partial \tilde{M} \cap C(u,r+2)$ and denote by $n$ the inward pointing normal to $\partial \tilde{M}$ with respect to the lift of $\mathfrak{h}$. Since $K$ is compact, there exists $\varepsilon > 0$, with $\varepsilon < 1/8$, such that the map $\Phi \colon K \times [0,2 \varepsilon) \to \tilde{M}$, defined by $\Phi(x,t) := \expo_{x}(tn)$ is a diffeomorphism onto its image $K_{\varepsilon}$.
Let $K_{1} := \partial \tilde{M} \cap C(u,r +1/2 + 2\varepsilon)$ and $K_{2} := \partial \tilde{M} \cap C(u,r+1 -2 \varepsilon)$. Clearly, there exists a non-negative $\tau \in C^{\infty}_{c}(\partial \tilde{M})$, with $\supp \tau \subset K_{2}$ and $\tau = 1$ in $K_{1}$. Extend it to $\tau^{\prime}$ in $K_{\varepsilon}$ by $\tau^{\prime}(\Phi(x,t)) := \tau(x)$, for all $(x,t) \in K \times [0,2\varepsilon)$. Consider a smooth $f \colon \mathbb{R} \to \mathbb{R}$, with $0 \leq f \leq 1$, $f(x) = 1$ for $x \leq \varepsilon$ and $f(x) = 0$ for $x \geq 3 \varepsilon/2$, and the function $\nu$ defined in $K_{\varepsilon}$ by $\nu(\Phi(x,t)) = f(t)$, for all $(x,t) \in K \times [0,2\varepsilon)$. Extend $\nu$ by zero outside $K_{\varepsilon}$ and set
$$
\psi_{u} := \nu \tau^{\prime} + (1-\nu) \psi_{u}^{\prime}.
$$
Since $\supp (\nu \tau^{\prime}) \subset C(u,r+1)$, $\supp \psi_{u}^{\prime} \subset C(u,r+1)$, it follows that $\supp \psi_{u} \subset C(u,r+1)$. Since $\varepsilon < 1/8$, the points where $\nu$ is not smooth are not in $C(u,r+1)$, which yields that $\psi_{u} \in C_{c}^{\infty}(\tilde{M})$. Since $\psi_{u}^{\prime} = 1$ in $C(u,r+1/2)$ and $\tau^{\prime} = 1$ in $C(u,r+1/2) \cap K_{\varepsilon}$, it follows that $\psi_{u} = 1$ in $C(u,r+1/2)$. In $\Phi(K \times [0,\varepsilon))$, which is a neighborhood of $\supp \psi_{u} \cap \partial \tilde{M}$, we have $\psi_{u} = \tau^{\prime}$. In particular, $\grad_{\mathfrak{h}} \psi_{u}$ is tangential to $\partial \tilde{M}$, and so is $\grad \psi_{u}$, since $\mathfrak{g}$ and $\mathfrak{h}$ are conformal. % it follows that $\grad \psi_{u}$ is also tangential to $\partial \tilde{M}$
\qed
\end{proof}
\medskip

Let $\psi_{u}$ be a function as in the above lemma and for any $y \in p^{-1}(x)$, fix $u(y) \in p_{2}^{-1}(y)$ and $g(y) \in \Gamma_{1}$, such that $u(y) = g(y)u$. Consider the functions $\psi_{u(y)} := \psi_{u} \circ g(y)^{-1}$ in $\tilde{M}$ and $\psi_{y}$ in $M_{2}$ defined by
\begin{equation}\label{Definition of pushed-down}
\psi_{y}(z) := \sum_{w \in p_{2}^{-1}(z)} \psi_{u(y)}(w).
\end{equation}
It is clear that $\psi_{y} \in C^{\infty}_{c}(M_{2})$, $\supp \psi_{y} \subset C(y,r+1)$ and $\psi_{y} \geq 1$ in $C(y,r + 1/2)$, for any $y \in p^{-1}(x)$.
%Since $\psi_{u(y)} \in C^{\infty}_{c}(\tilde{M})$, it follows that $\psi_{y} \in C^{\infty}_{c}(M_{2})$. 
Moreover, if $M_{1}$ has non-empty boundary, then $\grad \psi_{y}$ is tangential to $\partial M_{2}$, for all $y \in p^{-1}(x)$.
%It is clear that $\supp \psi_{y} \subset C(y,r+1)$ and $\psi_{y} \geq 1$ in $C(y,r + 1/2)$.
From Lemma \ref{Intersection Lemma}, there exists $N(r+2) \in \mathbb{N}$, such that for any $z \in M_{2}$, the ball $B(z,1)$ intersects at most $N(r+2)$ of the supports of $\psi_{y}$, with $y \in p^{-1}(x)$. Therefore, $\sum_{y \in p^{-1}(x)} \psi_{y}$ is locally a finite sum and hence, well-defined and smooth. 

If $M_{1}$ is compact, we choose $r$ large enough, so that $\sum_{y \in p^{-1}(x)} \psi_{y} \geq 1$ in $M_{2}$. In this case, set $\psi_{1} := 0$ in $M_{2}$.
If $M_{1}$ is non-compact, consider $f_{1} \in C_{c}^{\infty}(M_{1})$ with $0 \leq f_{1} \leq 1$, $f_{1} = 1$ in $C(x,r)$, $\supp f_{1} \subset B(x,r+1/2)$, and let $\psi_{1}$ be the lift of $1 - f_{1}$ on $M_{2}$. Then $\psi_{1} \in C^{\infty}(M_{2})$, $\psi_{1} \geq 0$ in $M_{2}$ and $\psi_{1} = 0$ in $C(y,r)$, for all $y \in p^{-1}(x)$. Evidently, $\psi_{1} + \sum_{y \in p^{-1}(x)} \psi_{y} \geq 1$ in $M_{2}$.

Consider the smooth partition of unity consisting of the functions
\begin{equation}\label{Partition of Unity}
\varphi_{1}  := \frac{\psi_{1}}{\psi_{1} + \sum_{y^{\prime} \in p^{-1}(x)} \psi_{y^{\prime}}} \, \text{ and } \,
 \varphi_{y}  := \frac{\psi_{y}}{\psi_{1} + \sum_{y^{\prime} \in p^{-1}(x)} \psi_{y^{\prime}}},
\end{equation}
with $y \in p^{-1}(x)$.

\begin{remark}\label{Tangential Gradient}
It is clear that $\supp \varphi_{1} = \supp \psi_{1}$, $\supp \varphi_{y} = \supp \psi_{y}$, $\sum_{y^{\prime} \in p^{-1}(x)} \varphi_{y^{\prime}} = 1$ in $C(y,r)$ and $\varphi_{y} > 0$ in $C(y,r+1/2)$, for any $y \in p^{-1}(x)$. If $M_{1}$ has non-empty boundary, then for any $y,y^{\prime} \in p^{-1}(x)$, we have that $\grad \psi_{y}$ is tangential to $\partial M_{2}$ and $\psi_{1} = 0$ in $B(y^{\prime},r)$. This yields that $\grad \varphi_{y}$ is tangential to $\partial M_{2}$ in $B(y^{\prime},r)$, for all $y,y^{\prime} \in p^{-1}(x)$.
\end{remark}

Let $\eta \in \mathcal{D}(D_{1})$ and $\theta \in \Gamma(E_{2})$ be the lift of $\eta$. Fix $x \in M_{1}$, $u \in p_{1}^{-1}(x)$ and $r>0$, such that $\supp \eta \subset B(x,r)$. If $M_{1}$ has non-empty boundary, we choose $r$ large enough, so that $B(u,r) \cap \partial \tilde{M} \neq \emptyset$. Consider a partition of unity associated to $x$, $u$ and $r$ as in (\ref{Partition of Unity}) and for a finite $P \subset p^{-1}(x)$, set $\chi := \sum_{y \in P} \varphi_{y}$.

\begin{remark}\label{Pull Up in the Domain}
Since $P$ is finite, it follows that $\chi \in C^{\infty}_{c}(M_{2})$ and $\chi \theta \in \Gamma_{c}(E_{2})$. Since $\supp \eta \subset B(x,r)$, we have that $\supp \theta$ is contained in the union of the balls $B(y,r)$, with $y \in p^{-1}(x)$. Therefore, if $M_{1}$ has non-empty boundary, from Remark \ref{Tangential Gradient}, $\chi \theta$ satisfies analogous boundary conditions to $\eta$, that is, $\chi \theta \in \mathcal{D}(D_{2})$.
\end{remark}

\begin{proposition}\label{Uniform Estimate}
There exists a constant $C$, independent from $P$, such that
for any $z \in M_{2}$, we have $\| D_{2}(\chi \theta)(z) \| \leq C$.
\end{proposition}

\begin{proof}
Consider $\delta > 0$, such that for any $x^{\prime} \in C(x,r+1)$, the ball $B(x^{\prime},2\delta)$ is evenly covered and contained in a coordinate neighborhood, and $E_{1}|_{B(x^{\prime},2\delta)}$ is trivial.
Let $x_{1},\dots,x_{k} \in C(x,r+1)$, such that the balls $B(x_{i}, \delta)$, with $1\leq i\leq k$, cover $C(x,r+1)$.
In any ball $B(x_{i},2 \delta)$, $D_{1}$ has a local expression of the form (\ref{Local Expression}), with $A^{\alpha}$ smooth. This yields that in $B(x_{i},\delta)$, $D_{1}$ is expressed in the form (\ref{Local Expression}), with $A^{\alpha}$ smooth and bounded.
For any such ball, we fix a coordinate system (which can be extended to the corresponding ball of radius $2 \delta$) and a trivialization. Since $C(x,r+1)$ is covered by finitely many such balls, it follows that there exists $C_{1}>0$, such that in any of these balls,
%for any of these coordinate systems and the corresponding trivialization, 
we have $\|A^{\alpha}\| \leq C_{1}$, for all multi-indices $\alpha$ of absolute value less or equal to the order $d$ of $D_{1}$.

Since $\eta$ is smooth and compactly supported in $B(x,r)$, there exists $C_{2}>0$, such that in any of these balls,
%for any of these coordinate systems and the corresponding trivialization, 
denoting by $(\eta^{(1)} , \dots ,\eta^{(\ell)})$ the local expression of $\eta$, we have that
$$
\| \frac{\partial^{|\alpha|}}{\partial x^{\alpha}} (\eta^{(1)} , \dots , \eta^{(\ell)}) \| \leq C_{2},
$$
for all multi-indices $\alpha$ of absolute value less or equal to $d$, that is, we have \textit{uniform estimates up to order $d$ for $\eta$ (with respect to this system of trivializations)}.
We lift these balls and the corresponding coordinate systems and trivializations to $M_{2}$ and $\tilde{M}$. Similarly, if $\psi_{1} \neq 0$, we obtain uniform estimates up to order $d$ for $f_{1}$, which yield uniform estimates up to order $d$ for $\psi_{1}$ (with respect to the lifted system on $M_{2}$).

Since $\psi_{u}$ is smooth and compactly supported in $C(u,r+1)$, which intersects finitely many balls of the lifted system on $\tilde{M}$, there exist uniform estimates up to order $d$ for $\psi_{u}$. Since $\psi_{u(y)}$ is a composition of $\psi_{u}$ with an element of $\Gamma_{1}$, we obtain the same uniform estimates up to order $d$ for $\psi_{u(y)}$, for all $u(y)$.
Recall the definition of $\psi_{y}$ in (\ref{Definition of pushed-down}). Consider a ball $B(z^{\prime},\delta)$ of the lifted system on $M_{2}$, which intersects $\supp \psi_{y}$, and the corresponding coordinate system.
%For $z \in \supp \psi_{y} \subset C(y,r+1)$, there exists a ball of the system containing $z$ and a corresponding trivialization. 
It is clear that for any $w \in p_{2}^{-1}(z^{\prime})$, the lifted system on $\tilde{M}$ contains the ball $B(w,\delta)$ and the corresponding coordinate system. From Lemma \ref{Cardinality Estimate Lemma}, there exists $\tilde{N}(r + 1,\delta) \in \mathbb{N}$, independent from $y$ and $z^{\prime}$, such that at most $\tilde{N}(r+1,\delta)$ such balls intersect the support of $\psi_{u(y)}$. Since we have uniform estimates up to order $d$ for $\psi_{u(y)}$, which are independent from $y \in p^{-1}(x)$, we obtain the same uniform estimates up to order $d$ for $\psi_{y}$, for all $y \in p^{-1}(x)$. From Lemma \ref{Intersection Lemma}, it follows that at most $N(r+1+ \delta)$ of the supports of $\psi_{y}$, with $y \in p^{-1}(x)$, intersect the open ball $B(z,\delta)$, for any $z \in M_{2}$. This yields that there exist uniform estimates up to order $d$ for $\psi_{1} + \sum_{y \in p^{-1}(x)} \psi_{y}$.

Recall the definition of $\varphi_{y}$ in (\ref{Partition of Unity}). Since the denominator is greater or equal to $1$ and we have uniform estimates (independent from $y$) up to order $d$ for the numerator and the denominator, we obtain the same uniform estimates up to order $d$ for $\varphi_{y}$, for all $y \in p^{-1}(x)$.
From Lemma \ref{Intersection Lemma}, at most $N(r+1+\delta)$ of the supports of $\varphi_{y}$, with $y \in p^{-1}(x)$, intersect the ball $B(z,\delta)$, for any $z \in M_{2}$. Therefore, we obtain uniform estimates up to order $d$ for $\chi$, which are independent from $P$

Clearly, for $z \in \supp (\chi \theta)$, we have that $z \in B(y,r)$, for some $y \in p^{-1}(x)$, and in particular, $z$ is contained in a ball of the system. With respect to the corresponding coordinate system and trivialization, denoting by $(\theta^{(1)},\dots,\theta^{(\ell)})$ the local expression of $\theta$, we have
%\[
%\| D_{2}(\chi \theta)(z) \| = \big\| \sum_{|\alpha|\leq d} A^{\alpha}(z^{\prime}) \frac{\partial^{|\alpha|}}{\partial x^{\alpha}} (\chi (\theta^{(1)},\dots , \theta^{(\ell)}  )) \big\| \leq C_{1}C_{2}C_{3} C(d,\ell),
%\]
\begin{eqnarray}
\| D_{2}(\chi \theta)(z) \| &=& \| \sum_{|\alpha|\leq d} (A^{\alpha} \circ p)(z) \frac{\partial^{|\alpha|}}{\partial y^{\alpha}} (\chi (\theta^{(1)},\dots , \theta^{(\ell)}  )) (z) \| \nonumber\\
&\leq&\sum_{|\alpha|\leq d} C_{1} \|\frac{\partial^{|\alpha|}}{\partial y^{\alpha}} (\chi (\theta^{(1)},\dots , \theta^{(\ell)}  )) (z) \| \nonumber \\
%&\leq& C_{1} C(d) \sum_{ | \alpha | \leq d } \sum_{\beta + \gamma = \alpha} \left| \frac{\partial^{|\beta|} \chi}{\partial x^{\beta}} (z^{\prime}) \right| \left\| \frac{\partial^{|\gamma|}}{\partial x^{\gamma}} (\theta^{(1)} , \dots, \theta^{(\ell)}) (z^{\prime}) \right\| \nonumber \\
%&\leq& C_{1} C_{3}C(d) \sum_{|\alpha| \leq d} \sum_{\beta + \gamma = \alpha} \left\| \frac{\partial^{|\gamma|}}{\partial x^{\gamma}} (\eta^{(1)} , \dots, \eta^{(\ell)}) (z^{\prime}) \right\| \nonumber \\
&\leq& C_{1}C_{2}C_{3}C(d,\ell), \nonumber
\end{eqnarray}
where 
%$\partial /\partial y_{i}$, $i=1,\dots,\text{dim}M_{1}$, are the vector fields induced from the lifted coordinate systems on $M_{2}$, 
$C_{3}$ is the uniform bound up to order $d$ for $\chi$ (which is independent from $P$) and $C(d,\ell)$ is a constant depending only on $d$ and $\ell$.
\qed
\end{proof}

\begin{corollary}\label{Rayleigh Quotient Uniform Estimate}
There exists a constant $C^{\prime}$, independent from $P$, such that for any $z \in M_{2}$, we have $| \langle D_{2}(\chi \theta)(z) , (\chi \theta)(z) \rangle | \leq C^{\prime}$.
\end{corollary}

\begin{proof}
Follows immediately from Proposition \ref{Uniform Estimate}.\qed
\end{proof}

\subsection{Amenable coverings}\label{spectrum subsection}

In this subsection we continue to work in the setting of the previous subsection, that is, we extend the covering $p \colon M_{2} \to M_{1}$ to a Riemannian covering $p \colon N_{2} \to N_{1}$ according to Proposition \ref{Extension of Covering} (if needed) and consider conformal Riemannian metrics, such that the manifolds become complete.
If $M_{1}$ has empty boundary, for $x \in M_{1}$ and $y \in p^{-1}(x)$, we denote by $D_{y}$ the fundamental domain of $p \colon M_{2} \to M_{1}$ centered at $y$, with respect to these conformal Riemannian metrics.
If $M_{1}$ has non-empty boundary, we denote by $D_{y}$ the part of the fundamental domain of $p \colon N_{2} \to N_{1}$ that lies in $M_{2}$. Furthermore, volumes, integrals and $L^{2}$-norms are with respect to the original Riemannian metrics.

As in the previous subsection, consider the universal coverings $p_{i} \colon \tilde{M} \to M_{i}$, denote by $\Gamma_{i}$ the deck transformations group of $p_i$, $i=1,2$, and fix $x \in M_{1}$ and $u \in p_{1}^{-1}(x)$. It is quite convenient to identify $\Gamma_{2} \backslash \Gamma_{1}$ with $p^{-1}(x)$, that is, $\Gamma_2 \gamma$ is identified with $p_2 (\gamma u)$, and study induced right action of $\Gamma_{1}$ on $p^{-1}(x)$. Clearly, if $y = p_{2}(\gamma u)$, for some $\gamma \in \Gamma_{1}$, then $y \cdot g = p_{2}(\gamma g u)$, for any $g \in \Gamma_{1}$. It is worth to point out that $p$ is amenable if and only if this right action of $\Gamma_{1}$ on $p^{-1}(x)$ is amenable.

For $r>0$, consider the finite set
\[
G_{r} := \{ g \in \Gamma_{1} : d(u,gu) < r \}
\]
and the subgroup $\langle G_{r} \rangle$ of $\Gamma_{1}$ generated by $G_{r}$. We are interested in the right action of $\langle G_{r} \rangle$ on $p^{-1}(x)$. The next remark is a simple description of the orbits of this action.

\begin{remark}\label{Orbits}
Let $y \in p^{-1}(x)$ and $g \in G_{r}$. Then there exists $\gamma \in \Gamma_{1}$, such that $y = p_{2}(\gamma u)$ and $y \cdot g = p_{2}(\gamma g u)$. Clearly, we have
\[
d(y, y \cdot g) = d(p_{2}(\gamma u) , p_{2}(\gamma gu)) \leq d(\gamma u , \gamma g u) = d(u, gu) < r.
\]
Conversely, let $y_{1} , y_{2} \in p^{-1}(x)$ with $d(y_{1}, y_{2}) < r$. Then there exist $\gamma_{1}, \gamma_{2} \in \Gamma_{1}$, such that $y_{i} = p_{2}(\gamma_{i} u)$, for $i=1,2$, and there exists $\sigma \in \Gamma_{2}$, such that
\[
d(\sigma \gamma_{1} u , \gamma_{2} u) = d (p_{2}(\gamma_{1}u) , p_{2}(\gamma_{2}u)) = d(y_{1},y_{2}) < r.
\]
This yields that $\gamma_{1}^{-1} \sigma^{-1} \gamma_{2} =:g \in G_{r}$. It follows that $\Gamma_{2} \gamma_{2} = \Gamma_{2} \gamma_{1} g$, that is, $y_{2} = y_{1} \cdot g$. 

Hence, two points $z_1,z_2 \in p^{-1}(x)$ are in the same orbit of the action of $\langle G_r \rangle$ on $p^{-1}(x)$ if and only if there exist $k \in \mathbb{N}$ and $y_1,\dots , y_k \in p^{-1}(x)$, such that $y_1 = z_1$, $y_k = z_2$ and $d(y_i,y_{i+1}) < r$, for $i=1,\dots, k-1$.
\end{remark}

\begin{lemma}\label{Orbit Cases}
If $p \colon M_{2} \to M_{1}$ is infinite sheeted, then there exists $R>0$, such that one of the following holds:
\begin{enumerate}[topsep=0pt,itemsep=-1pt,partopsep=1ex,parsep=0.5ex,leftmargin=*, label=(\roman*), align=left, labelsep=0em]
\item either for any $r \geq R$, the action of $\langle G_{r} \rangle$ on $p^{-1}(x)$ has only infinite orbits,
\item or for any $r \geq R$, the action of $\langle G_{r} \rangle$ on $p^{-1}(x)$ has infinitely many finite orbits.
\end{enumerate}
\end{lemma}

\begin{proof}
Assume to the contrary that the statement does not hold. Then there exists $r_{0} > 0$, such that the action of $\langle G_{r_{0}}\rangle$ on $p^{-1}(x)$ has only finitely many finite orbits $\mathcal{O}_{1}, \dots, \mathcal{O}_{k}$, for some $k \in \mathbb{N}$. Since $p$ is infinite sheeted, there exists also an infinite orbit $\mathcal{O}$. Since the action of $\Gamma_1$ on $p^{-1}(x)$ is transitive, for $y_{i} \in \mathcal{O}_{i}$, there exists $g_{i} \in \Gamma_1$, such that $y_i \cdot g_i \in \mathcal{O}$, for $i=1,\dots,k$. Then there exists $R>0$, such that $G_{r_{0}} \cup \{ g_1 , \dots , g_k\} \subset G_{R}$ and the action of $\langle G_{R} \rangle$ on $p^{-1}(x)$ has only infinite orbits. It is clear that so does the action of $\langle G_{r} \rangle$ on $p^{-1}(x)$, for any $r \geq R$, which is a contradiction. \qed
\end{proof}
\medskip

Let $r > 0$ large enough, so that $B(u,r) \cap \partial \tilde{M} \neq \emptyset$, if $M_{1}$ has non-empty boundary. If $p$ is infinite sheeted, we choose $r \geq R$, where $R$ is the constant from Lemma \ref{Orbit Cases}.
%If $p$ is infinite sheeted, let $r \geq R$ (from Lemma \ref{Orbit Cases}). If $M_{1}$ has non-empty boundary, we chose $r$ large enough, so that $B(u,r) \cap \partial \tilde{M} \neq \emptyset$.
Consider a partition of unity consisting of the functions $\varphi_{1}$ and $\varphi_{y}$, with $y \in p^{-1}(x)$, associated to $x$, $u$ and $r$ as in (\ref{Partition of Unity}). For a finite $P \subset p^{-1}(x)$, let $\chi := \sum_{y \in P} \varphi_{y}$ and consider the sets
\begin{eqnarray}\label{definitions of Q}
Q_{+} &:=& \{ y \in p^{-1}(x) : \chi = 1 \text{ in } B(y,r) \} \nonumber \\
Q_{-} &:=& \{ y \in p^{-1}(x) : 0 <  \chi(z) < 1  \text{ for some } z \in B(y,r) \},  \\
Q&:=& Q_{+} \cup Q_{-} = \{ y \in p^{-1}(x) : \text{$\chi(z)\neq 0$ for some $z\in B(y,r)$} \}. \nonumber
\end{eqnarray}
Clearly, $\chi = 0$ in $B(y,r)$, for any $y \in p^{-1}(x) \smallsetminus Q$. Since $\chi$ is compactly supported, it follows that $Q$ is finite.
%Note that for any $y \in p^{-1}(x) \smallsetminus Q$, we have $\chi = 0$ in $B(y,r)$.
%Furthermore, $Q$ is finite, since $\chi$ is compactly supported. 
The proof of the following lemma is essentially presented in \cite{BMP}, but since we are in a different situation here, we repeat it.

\begin{lemma}\label{Pulling Up}
If $p$ is amenable, then for any $\varepsilon > 0$, there exists a non-empty, finite $P \subset p^{-1}(x)$, such that $$\frac{\#(Q_{-})}{ \#(Q_{+})} < \varepsilon.$$
\end{lemma}

\begin{proof}
From Proposition \ref{Folner}, since $p$ is amenable, for any $\delta > 0$, there exists a non-empty, finite $P \subset p^{-1}(x)$, such that
\[
\#(P \smallsetminus Pg) < \delta \#(P),
\]
for all $g \in G_{2r+2}$. From Remark \ref{Tangential Gradient}, we have that $\supp \varphi_{y_{0}} \subset C(y_{0},r+1)$, $\varphi_{y_{0}} > 0$ in $B(y_{0},r + 1/2)$ and $\sum_{y \in p^{-1}(x)} \varphi_{y} = 1$ in $B(y_{0} , r)$, for any $y_{0} \in p^{-1}(x)$. Clearly, $P$ is contained in $Q$, which implies that $\#(P) \leq \#(Q)$.

For $y \in Q_{-}$, there exists $z \in B(y,r)$, such that $0<\chi(z)<1$. Therefore, there exist $y_{1} \in P$ and $y_{2} \in p^{-1}(x) \smallsetminus P$, such that $\varphi_{y_{i}}(z) > 0$, which yields that $d(y_{i},z)<r+1$, for $i=1,2$.
%$d(y,y_{i}) < 2r + 1$, for $i=1,2$.
It follows that $d(y_{1}, y_{2}) < 2r +2$ and from Remark \ref{Orbits}, there exists $g \in G_{2r + 2}$, such that $y_{1} = y_{2} \cdot g$. In particular, $y_{1} \in P \smallsetminus P g$. Since $d(y,y_{1}) < 2r+1$, from Lemma \ref{Intersection Lemma}, for a fixed $y_{1}$, there exist at most $N(2r+1)$ such $y$. Since $y_{1} \in P \smallsetminus P g$, for some $g \in G_{2r+2}$, there exist at most $\delta \#(P) \#(G_{2r+2})$ such $y_{1}$. Hence, it follows that
\[
\#(Q_{-}) \leq \delta  \#(P) \#(G_{2r+2})N(2r+1) \leq \delta \#(Q) \#(G_{2r+2}) N(2r+1).
\]
Since $Q$ is the disjoint union of $Q_{+}$ and $Q_{-}$,
%Since $\#(Q) = \#(Q_{-}) + \#(Q_{+})$, 
for $\delta \#(G_{2r+2}) N(2r+1) < 1$, we have
\[
\frac{\#(Q_{-})}{\#(Q_{+})} \leq \frac{\delta \#(G_{2r+2})N(2r+1)}{1 - \delta \#(G_{2r+2})N(2r+1)}.
\]
This completes the proof, since $\delta > 0$ is arbitrarily small. \qed
\end{proof}

\begin{proposition}\label{Pulling Up Strong Version}
If $p \colon M_{2} \to M_{1}$ is infinite sheeted and amenable, then for any $\varepsilon > 0$ and $K \subset M_{2}$ compact, there exists a non-empty, finite $P \subset p^{-1}(x)$, such that $\supp \chi$ does not intersect $K$ and $$\frac{\#(Q_{-})}{ \#(Q_{+})} < \varepsilon.$$
\end{proposition}

\begin{proof}
First assume that the second statement of Lemma \ref{Orbit Cases} holds. Then the action of $\langle G_{2r+2} \rangle$ on $p^{-1}(x)$ has infinitely many finite orbits $\mathcal{O}_{n}$, with $n \in \mathbb{N}$. Clearly, for $P := \mathcal{O}_{n}$, we have that $Q_{-}$ is empty. %$= \emptyset$. 
Indeed, if there exists $y_{0} \in Q_{-}$, then there exist $z \in B(y_{0},r)$, $y_{1} \in P$ and $y_{2} \in p^{-1}(x) \smallsetminus P$, such that $\varphi_{y_{i}} (z) > 0$, $i=1,2$. It follows that $d(z,y_{i}) < r + 1$, $i=1,2$, which yields that $d(y_{1},y_{2}) < 2r+2$. From Remark \ref{Orbits}, there exists $g \in G_{2r+2}$, such that $y_{2} = y_{1} \cdot g$, which is a contradiction, since $P$ is an orbit of the action of $\langle G_{2r+2} \rangle$ on $p^{-1}(x)$.

For a compact $K \subset M_{2}$, the set $P_{K} := p^{-1}(x) \cap B(K, r +2)$ is finite and in particular, intersects only finitely many orbits $\mathcal{O}_{n}$. Let $P$ be an orbit that does not intersect $P_{K}$. Since $\supp \varphi_{y} \subset C(y,r+1)$, for any $y \in p^{-1}(x)$, it is clear that for such $P$, 
the support of $\chi$ does not intersect $K$.
%we have $\supp \chi \cap K = \emptyset$.

If the first statement of Lemma \ref{Orbit Cases} holds, then the action of $\langle G_{r} \rangle$ on $p^{-1}(x)$ has only infinite orbits. For a compact $K \subset M_{2}$, consider the finite set $P_{K} := p^{-1}(x) \cap B(K, r +2)$.
From Lemma \ref{Pulling Up}, for any $\varepsilon > 0$, there exists a non-empty, finite $P \subset p^{-1}(x)$, such that
\[
\frac{\#(Q_{-})}{\#(Q_{+})} < \delta := \frac{\varepsilon}{1 + (1 +\varepsilon)N(2r+1)\#(P_{K})},
\]
where $N(2r+1)$ is the constant from Lemma \ref{Intersection Lemma}.

Since the action of $\langle G_{r} \rangle$ on $p^{-1}(x)$ has only infinite orbits, it follows that $Q_{-}$ is non-empty. %$\neq \emptyset$. 
Indeed, since $P$ is non-empty and this action has only infinite orbits, there exists an infinite orbit $\mathcal{O}$ and $z_1 \in P \cap \mathcal{O}$. Since $P$ is finite, there exists $z_2 \in \mathcal{O} \smallsetminus P$, and
%Since $z_1, z_2$ are in the same orbit, 
from Remark \ref{Orbits}, there exist $k \in \mathbb{N}$ and $y_1, \dots, y_k \in p^{-1}(x)$, with $y_1 = z_1$, $y_k = z_2$ and $d(y_i,y_{i+1}) < r$, for $i=1,\dots, k-1$. Since $y_1 \in P$ and $y_k \notin P$, there exists $1 \leq j < k$, such that $y_j  \in P$ and $y_{j+1} \notin P$. Since $d(y_j, y_{j+1}) < r$, it follows that $0<\chi(y_{j+1})<1$ and in particular, $y_{j} \in Q_{-}$.

%Since $Q_{+} \subset P$ and $Q_{-} \neq \emptyset$, it is clear that
Evidently, $Q_{+}$ is contained in $P$. Since $Q_{-}$ is non-empty, it is clear that
\[
\frac{1}{\delta} \leq \#(Q_{+}) \leq \#(P),
\]
which yields that $\#(P) > \#(P_{K})$, from the choice of $\delta$. In particular, the finite set $P^{\prime} := P \smallsetminus P_{K}$ is non-empty. Consider the function $\chi^{\prime}$ and the sets $Q_{+}^{\prime}$, $Q_{-}^{\prime}$ and $Q^{\prime}$ corresponding to $P^{\prime}$ as in (\ref{definitions of Q}). Clearly, the support of $\chi^{\prime}$ does not intersect $K$, since $\supp \varphi_{y} \subset C(y,r+1)$, for any $y \in p^{-1}(x)$.

From Lemma \ref{Intersection Lemma}, it follows that for any $y_{0} \in p^{-1}(x)$, the support of $\varphi_{y_{0}}$ intersects at most $N(2r+1)$ open balls $B(y,r)$, with $y \in p^{-1}(x)$. Hence, we have that
\begin{eqnarray}
\#(Q_{-}^{\prime}) &\leq& \#(Q_{-}) + N(2r+1) \#(P_{K}), \nonumber \\
\#(Q_{+}^{\prime}) &\geq& \#(Q_{+}) - N(2r+1) \#(P_{K}). \nonumber
\end{eqnarray}
Therefore, we obtain
\[
\frac{\#(Q_{-}^{\prime})}{\#(Q_{+}^{\prime})} \leq \frac{\#(Q_{-}) + N(2r+1) \#(P_{K})}{\#(Q_{+}) - N(2r+1) \#(P_{K})} < \varepsilon,
\]
from the choice of $\delta$. \qed
\end{proof}

\begin{remark}\label{integral remark}
After endowing $M_{1}$ or $N_{1}$ with $\mathfrak{h}$ (depending on whether $M_{1}$ has empty boundary or not) and the covering space with its lift, we have that $p \colon D_{y} \to M_{1}$ is an isometry up to sets of measure zero, for any $y \in p^{-1}(x)$. Thus, for $f \in C_{c}(M_{1})$, we have
\begin{equation}\label{Integral}
\int_{D_{y}} (f \circ p) d {\Vol}_{\mathfrak{h}_{2}} = \int_{M_{1}} f d {\Vol}_{\mathfrak{h}_{1}},
\end{equation}
where $\Vol_{\mathfrak{h}_{i}}$ (respectively, $\Vol_{\mathfrak{g}_{i}}$) is the measure on $M_{i}$ induced by $\mathfrak{h}$ (respectively, $\mathfrak{g}$) or its lift, $i=1,2$.
Since $\mathfrak{g}$ and $\mathfrak{h}$ are conformal, there exists a positive $\varphi_{v} \in C^{\infty}(M_{1})$, such that 
$$ \frac{d \Vol_{\mathfrak{h}_{1}}}{d\Vol_{\mathfrak{g}_{1}}} = \varphi_{v}  \text{ and } \frac{d \Vol_{\mathfrak{h}_{2}}}{d\Vol_{\mathfrak{g}_{2}}} = \varphi_{v} \circ p .$$
For simplicity of notation, we omit $d\Vol_{\mathfrak{g}_{i}}$ in the integrals and the index of $\Vol_{\mathfrak{g}_{i}}$.
From (\ref{Integral}), we have
$\int_{D_{y}} f \circ p = \int_{M_{1}} f$, for any $f \in C_{c}(M_{1})$ and $y \in p^{-1}(x)$.
Similarly, for a compact $K \subset M_{1}$, we have $\Vol(K) = \Vol(p^{-1}(K) \cap D_{y})$, for any $y \in p^{-1}(x)$.
\end{remark}

\begin{proposition}\label{Pull Up Lemma}
Let $p \colon M_2 \to M_1$ be an infinite sheeted, amenable Riemannian covering. Consider $\eta \in \mathcal{D}(D_1)$ with $\| \eta \|_{L^{2}(E_{1})}=1$ and $\lambda \in \mathbb{F}$. Then for any $\varepsilon>0$ and $K \subset M_{2}$ compact, there exists $\zeta \in \mathcal{D}(D_2)$ with $\| \zeta \|_{L^{2}(E_{2})} = 1$, such that $\supp \zeta \subset p^{-1}(\supp \eta)$, $\supp \zeta \cap K = \emptyset$ and $\| (D_{2} - \lambda) \zeta \|_{L^{2}(E_{2})} \leq \| (D_1 - \lambda) \eta \|_{L^{2}(E_{1})} + \varepsilon$.
\end{proposition}

\begin{proof}
%As in the beginning of this section, consider the universal covering 
Let $p_{1} \colon \tilde{M} \to M_{1}$ be the universal covering of $M_{1}$ and fix $x\in M_1$, $u \in p_{1}^{-1}(x)$ and $r \geq R$ (from Lemma \ref{Orbit Cases}), such that $\supp \eta \subset B(x,r)$ and $B(u,r) \cap \partial \tilde{M} \neq \emptyset$, if $M_{1}$ has non-empty boundary. %consider $r$ large enough, so that $B(u,r) \cap \partial \tilde{M} \neq \emptyset$.
Consider a partition of unity consisting of the functions $\varphi_{1}$ and $\varphi_{y}$, with $y \in p^{-1}(x)$, associated to $x$, $u$ and $r$ as in (\ref{Partition of Unity}), and let $\theta$ be the lift of $\eta$. From Remark \ref{Pull Up in the Domain}, for any finite set $P^{\prime} \subset p^{-1}(x)$ and $\chi^{\prime} := \sum_{y \in P^{\prime}} \varphi_{y}$, we have that $\chi^{\prime}\theta \in \mathcal{D}(D_{2})$. From Proposition \ref{Uniform Estimate}, there exists $C > 0$, independent from $P^{\prime}$, such that $\| D_{2}(\chi^{\prime} \theta)(z) \| \leq C$,
for any $z \in M_{2}$. Hence, we obtain that
\[
\max_{z \in M_{2}} \left\| (D_{2} - \lambda)(\chi^{\prime} \theta)(z) \right\| \leq C + |\lambda| \max_{w \in M_{1}} \| \eta(w) \| = : C_{0}.
\]
From Proposition \ref{Pulling Up Strong Version}, there exists a non-empty, finite $P \subset p^{-1}(x)$, such that the support of $\chi := \sum_{y \in P} \varphi_{y}$ does not intersect $K$ and
\[
\frac{\#(Q_{-})}{\#(Q_{+})} < \min\left\{ \frac{\varepsilon}{ C_{0}^{2} \Vol (\supp \eta)} , \varepsilon \right\},
\]
where $Q_{+}$, $Q_{-}$ and $Q$ are the sets corresponding to $P$ as in (\ref{definitions of Q}).

Since $\chi \theta$ is in the domain of $D_{2}$, so is the normalized section $\zeta := (1/\| \chi \theta \|_{L^{2}(E_{2})}) \chi \theta$.
Evidently, $\| \zeta \|_{L^{2}(E_{2})}=1$ and $\supp \zeta \subset p^{-1}(\supp \eta)$. 
From Lemma \ref{preimage of compact}, we have that $\supp \zeta \cap D_{y} \subset B(y,r)$, for any $y \in p^{-1}(x)$, which yields 
that $\supp \zeta$ is contained in the union of the fundamental domains $D_{y}$, with $y \in Q$.
%$\supp \zeta \subset \cup_{y \in Q} D_{y}$.
Clearly, we have
\[
\| \chi \theta \|_{L^{2}(E_{2})}^{2} \geq \sum_{y \in Q_{+}} \int_{D_{y}} \| \chi \theta\|^{2} =\sum_{y \in Q_{+}} \int_{D_{y}} \| \theta \|^{2} = \#(Q_{+}),
\]
from the definition of $Q_{+}$ and Remark \ref{integral remark}. Therefore, we obtain that
\begin{eqnarray}
\int_{M_{2}} \| (D_{2} - \lambda) \zeta \|^{2} &\leq& \frac{1}{\#(Q_{+})} \sum_{y \in Q_{+}} \int_{D_{y}} \| (D_{2} - \lambda) (\chi \theta) \|^{2} \nonumber \\
&+& \frac{1}{\#(Q_{+})} \sum_{y \in Q_{-}} \int_{D_{y}} \| (D_{2} - \lambda) (\chi \theta) \|^{2}. \nonumber
\end{eqnarray}
For $y \in Q_{+}$, we have $\chi = 1$ in $B(y,r)$, which is a neighborhood of $\supp \theta \cap D_{y}$. This implies that
\[
\frac{1}{\#(Q_{+})} \sum_{y \in Q_{+}} \int_{D_{y}} \| (D_{2} - \lambda) (\chi \theta) \|^{2} = \frac{1}{\#(Q_{+})} \sum_{y \in Q_{+}} \int_{D_{y}} \| (D_{2} - \lambda) \theta \|^{2} = \int_{M_{1}} \| (D_{1} - \lambda) \eta \|^{2}.
\]
Since $\| (D_{2} - \lambda) (\chi \theta)(z) \| \leq C_{0}$, for any $z \in M_{2}$, it follows that
\begin{eqnarray}
\frac{1}{\#(Q_{+})} \sum_{y \in Q_{-}} \int_{D_{y}} \| (D_{2} - \lambda) (\chi \theta) \|^{2} &\leq& \frac{C_{0}^{2}}{ \#(Q_{+}) } \sum_{y \in Q_{-}} \Vol (\supp \theta \cap D_{y}) \nonumber \\
&=& \frac{\#(Q_{-})}{\#(Q_{+})} C_{0}^{2} \Vol (\supp \eta) \leq \varepsilon. \nonumber
\end{eqnarray}
Hence, $\| (D_{2} - \lambda) \zeta \|_{L^{2}(E_{2})}^{2}  \leq \| (D_{1} - \lambda) \eta 
\|_{L^{2}(E_{1})}^{2} + \varepsilon$.
\qed
\end{proof}
% \medskip

%Similarly, using Corollary \ref{Rayleigh Quotient Uniform Estimate} instead of Proposition \ref{Uniform Estimate}, we obtain the following:

\begin{proposition}\label{Rayleigh Quotient Pulling Up}
Let $p \colon M_2 \to M_1$ be an infinite sheeted, amenable Riemannian covering, and assume that the operators $D_{i}$ are symmetric, $i=1,2$. Then for any section $\eta \in \mathcal{D}(D_1) \smallsetminus \{0\}$, $\varepsilon>0$ and $K \subset M_{2}$ compact, there exists $\zeta \in \mathcal{D}(D_2) \smallsetminus\{0\}$, such that $\supp \zeta \subset p^{-1}(\supp \eta)$, $\supp \zeta \cap K = \emptyset$ and $ \mathcal{R}_{D_{2}}(\zeta) \leq  \mathcal{R}_{D_1} (\eta)  + \varepsilon$.
\end{proposition}

\begin{proof}
The proof is similar to the proof of Proposition \ref{Pull Up Lemma}, using Corollary \ref{Rayleigh Quotient Uniform Estimate} instead of Proposition \ref{Uniform Estimate}. \qed
\end{proof}\medskip

%We are ready to prove the first of our main results. \medskip

\noindent{\emph{Proof of Theorem \ref{Inclusion of Spectrums}:}} Let $\lambda \in \sigma_{\ap}(\overline{D}_{1})$. From Lemma \ref{Spectrum of Closure}, there exists $(\eta_{k})_{k \in \mathbb{N}} \subset \mathcal{D}(D_{1})$, such that $\| \eta_{k} \|_{L^{2}(E_{1})} = 1$ and $(D_{1} - \lambda) \eta_{k} \rightarrow 0$ in $L^{2}(E_{1})$. Consider an exhausting sequence $(K_{k})_{k \in \mathbb{N}}$ of $M_{2}$. From Proposition \ref{Pull Up Lemma}, for any $k \in \mathbb{N}$, there exists $\zeta_{k} \in \mathcal{D}(D_{2})$, such that $\| \zeta_k \|_{L^{2}(E_{2})} = 1$, $\| (D_{2} - \lambda) \zeta_k \|_{L^{2}(E_{2})} \leq \| (D_1 - \lambda) \eta_k \|_{L^{2}(E_{1})} + 1/k$ and $\supp \zeta_k \cap K_k = \emptyset$. Therefore, $(D_2 - \lambda) \zeta_k \rightarrow 0$ in $L^{2}(E_{2})$ and for any compact $K \subset M_{2}$, there exists $k_{0} \in \mathbb{N}$, such that $\supp \zeta_{k} \cap K = \emptyset$, for all $k \geq k_{0}$. 
%$\supp \zeta_n$ eventually leave every compact subset of $M_{2}$.
It follows that $(\zeta_{k})_{k \in \mathbb{N}}$ is a Weyl sequence for $D^{\prime}_{2}$ and $\lambda$, and in particular, $\lambda \in \sigma_{W}(D^{\prime}_{2})$. \qed
\medskip

\noindent{\emph{Proof of Theorem \ref{Inclusion of Spectrums self-adj}:}} Follows immediately from Theorem \ref{Inclusion of Spectrums} and Proposition \ref{Spectrum of Self-adj}. \qed \medskip

Assume now that the operator $D_{i} \colon \mathcal{D}(D_{i}) \subset L^{2}(E_i) \to L^{2}(E_i)$ is symmetric and bounded from below, and let $D_{i}^{(F)}$ be its Friedrichs extension, $i=1,2$. For more details on the Friedrichs extension of a symmetric, bounded from below and densely defined linear operator on a Hilbert space, see \cite{MR2744150}. It is well-known that the Friedrichs extension of an operator preserves its lower bound. In particular, for $i=1,2$, we have
\begin{equation}\label{Bottom Friedrichs Extensions}
\lambda_{0}(D_{i}^{(F)}) = \inf_{\eta \in \mathcal{D}(D_{i}) \smallsetminus \{0\}} \mathcal{R}_{D_{i}}(\eta).
\end{equation}
Recall the following proposition for the essential spectrum of a self-adjoint operator.
\begin{proposition}[{\cite[Proposition 2.1]{MR592568}}]\label{Donelly}
Let $A \colon \mathcal{D}(A) \subset \mathcal{H} \to \mathcal{H}$ be a self-adjoint operator on a separable Hilbert space $\mathcal{H}$, and consider $\lambda \in \mathbb{R}$.
Then the interval $(-\infty, \lambda]$ intersects the essential spectrum of $A$ if and only if for any $\varepsilon > 0$, there exists an infinite dimensional subspace $G_{\varepsilon} \subset \mathcal{D}(A)$, such that $\mathcal{R}_{A}(v) < \lambda + \varepsilon$, for all $v \in G_{\varepsilon} \smallsetminus \{0\}$.
\end{proposition}

\noindent{\emph{Proof of Theorem \ref{Friedrichs}:}} From (\ref{Bottom Friedrichs Extensions}), it follows that there exists $(\eta_{k})_{k \in \mathbb{N}} \subset \mathcal{D}(D_1) \smallsetminus \{0\}$, such that $\mathcal{R}_{D_{1}}(\eta_k) \leq \lambda_{0}(D_{1}^{(F)}) + 1/k$, for any $k \in \mathbb{N}$. From Proposition \ref{Rayleigh Quotient Pulling Up}, there exists $(\zeta_k)_{k \in \mathbb{N}} \subset \mathcal{D}(D_2) \smallsetminus \{0\}$, such that $\mathcal{R}_{D_{2}}(\zeta_k) \leq \lambda_{0}(D_{1}^{(F)}) + 2/k$ and $\supp \zeta_{k} \cap \supp \zeta_{k^{\prime}} = \emptyset$, for all $k,k^{\prime} \in \mathbb{N}$, with $k \neq k^{\prime}$. Evidently, for any $\varepsilon > 0$, there exists $k_{0} \in \mathbb{N}$, such that $\mathcal{R}_{D_{2}}(\zeta_k) < \lambda_{0}(D_{1}^{(F)}) + \varepsilon$, for all $k \geq k_0$. Consider the subspace $G_{\varepsilon}$ of $\mathcal{D}(D_{2})$ spanned by $\{\zeta_{k} : k \geq k_{0} \}$. Since the sections $\zeta_k$, with $k \in \mathbb{N}$, have disjoint supports, the space $G_{\varepsilon}$ is infinite dimensional. 
Clearly, any $\theta \in G_{\varepsilon}$ is of the form $\theta := \sum_{i=k_0}^{k_0+\mu} m_i \zeta_i$, for some $\mu \in \mathbb{N}$ and $m_{k_{0}},\dots,m_{k_{0}+\mu} \in \mathbb{F}$. Therefore, we have
\[
\mathcal{R}_{D_2}(\theta) = \frac{\sum_{i=k_0}^{k_0 +\mu} |m_{i}|^{2} \langle D_2 \zeta_i , \zeta_i \rangle_{L^{2}(E_{2})}}{\sum_{i=k_0}^{k_0 + \mu} |m_{i}|^{2} \| \zeta_i \|^{2}_{L^{2}(E_{2})}} \leq \max_{k_{0} \leq k \leq k_{0} + \mu} \mathcal{R}_{D_{2}}(\zeta_i) < \lambda_{0}(D_{1}^{(F)}) + \varepsilon.
\]
%\[
%\mathcal{R}_{D_2}(\theta) = \frac{\langle D_{2} \theta , \theta \rangle_{L^{2}(E_{2})}}{\| \theta \|^{2}_{L^{2}(E_{2})}} = \frac{\sum_{i=k_0}^{k_0 +\mu} |m_{i}|^{2} \langle D_2 \zeta_i , \zeta_i \rangle_{L^{2}(E_{2})}}{\sum_{i=k_0}^{k_0 + \mu} |m_{i}|^{2} \| \zeta_i \|^{2}_{L^{2}(E_{2})}} \leq \max_{k_{0} \leq k \leq k_{0} + \mu} \mathcal{R}_{D_{2}}(\zeta_i) < \lambda_{0}(D_{1}^{(F)}) + \varepsilon.
%\]
%\[
%\mathcal{R}_{D_2}(\theta) = \frac{\langle D_{2} \theta , \theta \rangle}{\langle \theta , \theta \rangle} = \frac{\sum_{i=k_0}^{k_0 +\mu} |m_{i}|^{2} \langle D_2 \zeta_i , \zeta_i \rangle}{\sum_{i=k_0}^{k_0 + \mu} |m_{i}|^{2} \| \zeta_i \|^{2}} \leq \max_{i=k_{0}, \dots , k_0 + \mu} \mathcal{R}_{D_{2}}(\zeta_i) < \lambda_{0}(D_{1}^{(F)}) + \varepsilon,
%\]
%where inner products and norms are in $L^{2}(E_{2})$.
From Proposition \ref{Donelly},
it follows that $\lambda^{\ess}_{0}(D_2^{(F)}) \leq \lambda_{0}(D_{1}^{(F)})$. \qed

\begin{remark}
In the proof of Theorem \ref{Friedrichs}, the only properties of the Friedrichs extension used are self-adjointness and the preservation of the lower bound of $D_{1}$. Therefore, this proof establishes the analogous result for any self-adjoint extensions of the operators, as long as the extension of $D_{1}$ preserves its lower bound.
\end{remark}
%Note that this proof establishes the analogous result for any self-adjoint extensions of the operators, as long as the extension of $D_{1}$ preserves its lower bound.

For sake of completeness, we also present the analogous results for finite sheeted coverings. It is clear that they cannot be improved in order to obtain as strong statements as in the case of infinite sheeted amenable coverings.
%We omit their proofs, since they are elementary, taking into account Proposition \ref{Spectrum of Closure} and (\ref{Bottom Friedrichs Extensions}). It is important to observe that if $p$ is finite sheeted, then for a section $\eta$ in the domain of $D_{1}$, its lift $\theta$ is in the domain of $D_{2}$.

\begin{proposition}\label{Inclusion of Spectra finite sheeted}
Let $D_{2}^{\prime}$ be a closed extension of $D_{2}$. If $p$ is a finite sheeted Riemannian covering, then $\sigma_{\ap}(\overline{D}_{1}) \subset \sigma_{\ap}(D_{2}^{\prime})$ and $\sigma_{W}(\overline{D}_{1}) \subset \sigma_{W}(D_{2}^{\prime})$.
\end{proposition}

\begin{proof}
If $\eta$ is in the domain of $D_{1}$, then its lift is in the domain of $D_{2}$. For $\lambda \in \sigma_{W}(\overline{D}_{1})$, from Lemma \ref{Spectrum of Closure}, there exists a Weyl sequence $(\eta_{k})_{k \in \mathbb{N}} \subset \mathcal{D}(D_{1})$ for $\overline{D}_{1}$ and $\lambda$. Then, the sequence consisting of the normalized (in $L^{2}(E_{2})$) lifts of $\eta_{k}$, $k \in \mathbb{N}$, is a Weyl sequence for $D_{2}^{\prime}$ and $\lambda$. Hence, $\sigma_{W}(\overline{D}_{1}) \subset \sigma_{W}(D_{2}^{\prime})$. Similarly, it follows that $\sigma_{\ap}(\overline{D}_{1}) \subset \sigma_{\ap}(D_{2}^{\prime})$.\qed
\end{proof}

\begin{proposition}\label{Comparison of Bottoms finite sheeted}
Assume that $D_{i}$ is symmetric and bounded from below, and denote by $D_{i}^{(F)}$ its Friedrichs extension, $i=1,2$. If $p$ is a finite sheeted Riemannian covering, then $\lambda_{0}(D_{2}^{(F)}) \leq \lambda_{0}(D_{1}^{(F)})$.
\end{proposition}

\begin{proof}
If $\eta$ is in the domain of $D_{1}$, then its lift $\theta$ is in the domain of $D_{2}$. If $\eta \neq 0$, it is easy to see that $\mathcal{R}_{D_{1}}(\eta) = \mathcal{R}_{D_{2}}(\theta)$, and the statement follows from (\ref{Bottom Friedrichs Extensions}).\qed
\end{proof}\medskip

In the rest of this section, we give applications of our results in the case of Schr\"{o}dinger operators.
The following proposition characterizes the bottom of the spectrum of a Schr\"{o}dinger operator as the maximum of its positive spectrum.

\begin{proposition}\label{Positive Spectrum}
Let $S$ be a Schr\"{o}dinger operator on a Riemannian manifold $M$.
%Let $M$ be a Riemannian manifold with (possibly empty) smooth boundary and $S$ a Schr\"{o}dinger operator on $M$.
Then the bottom of the spectrum of $S$ is the maximum of all $\lambda \in \mathbb{R}$, such that there exists $\varphi \in C^{\infty}(M \smallsetminus \partial M)$ with $S\varphi = \lambda \varphi$, which is positive in $M \smallsetminus \partial M$. % in the interior of $M$.
\end{proposition}

\begin{proof}
If $M$ has empty boundary, then the statement may be found in \cite[Theorem 7]{MR0385749}, \cite[Theorem 1]{MR562550} and \cite[Theorem 2.1]{MR882827}. If $M$ has non-empty boundary, it is clear that $\lambda_{0}(S) = \lambda_{0}(S,M \smallsetminus \partial M)$, where $\lambda_{0}(S,M \smallsetminus \partial M)$ stands for the bottom of the spectrum of $S$ on the interior of $M$. Hence, in this case, the claim follows from the corresponding statement for manifolds without boundary. \qed
%It can be extended easily to manifolds with boundary, since for a Schr\"{o}dinger operator $S$ on a Riemannian manifold $M$ with non-empty boundary, we have $\lambda_{0}(S) = \lambda_{0}(S,M \smallsetminus \partial M)$, where $\lambda_{0}(S,M \smallsetminus \partial M)$ stands for the bottom of the spectrum of $S$ on the interior of $M$.\qed
\end{proof}\medskip

In particular, there exists $\varphi \in C^{\infty}(M \smallsetminus \partial M)$ with $S\varphi = \lambda_{0}(S) \varphi$, which is positive in the interior of $M$. 
It is worth to point out that the smooth eigenfunctions of the preceding proposition do not have to be square-integrable. %Moreover, note that there exists $\varphi \in C^{\infty}(M \smallsetminus \partial M)$ with $S\varphi = \lambda_{0}(S) \varphi$, which is positive in the interior of $M$. 
%Furthermore, given a Riemannian covering, the lift of an eigenfunction of a Schr\"{o}dinger operator is an eigenfunction of the lift of this Schr\"{o}dinger operator. 
The following corollary is a consequence of Proposition \ref{Positive Spectrum} (an alternative proof can be found in \cite{BMP}).
\begin{corollary}\label{Inequality of bottoms}
Let $p \colon M_{2} \to M_{1}$ be a Riemannian covering. Let $S_{1}$ be a Schr\"{o}dinger operator on $M_{1}$ and $S_{2}$ its lift on $M_{2}$. Then $\lambda_{0}(S_{1}) \leq \lambda_{0}(S_{2})$.
\end{corollary}

\begin{proof}
Follows immediately from Proposition \ref{Positive Spectrum}, since the lift of an eigenfunction of $S_{1}$ is an eigenfunction of $S_{2}$. \qed
\end{proof}

%The following corollaries are immediate consequences of Theorems \ref{Inclusion of Spectrums self-adj} and \ref{Friedrichs}, Propositions \ref{Inclusion of Spectra finite sheeted} and \ref{Comparison of Bottoms finite sheeted} and Corollary \ref{Inequality of bottoms}.

\begin{corollary}\label{incusion of spectra Schrodinger}
Let $p : M_{2} \rightarrow M_{1}$ be an infinite sheeted, amenable Riemannian covering. Let $S_{1}$ be a Schr\"{o}dinger operator on $M_{1}$ and $S_{2}$ its lift on $M_{2}$. Then $\lambda_{0}(S_{1}) = \lambda_{0}^{\ess}(S_{2})$. If, in addition, $M_{1}$ is complete, then $\sigma(S_{1}) \subset \sigma_{\ess}(S_{2})$.
\end{corollary}

\begin{proof}
Follows from Theorems \ref{Inclusion of Spectrums self-adj}, \ref{Friedrichs} and Corollary \ref{Inequality of bottoms}. \qed
\end{proof}\medskip

The following results describe the behavior of the spectrum of Schr\"{o}dinger operators under finite sheeted coverings.

%For sake of completeness, we also present the corresponding results for the spectrum of Schr\"{o}dinger operators under finite sheeted coverings.%, which should be already well-known.

\begin{corollary}\label{finite sheeted schroedinger}
Let $p \colon M_{2} \to M_{1}$ be a finite sheeted Riemannian covering. Let $S_{1}$ be a Schr\"{o}dinger operator on $M_{1}$ and $S_{2}$ its lift on $M_{2}$. Then $\lambda_{0}(S_{1}) = \lambda_{0}(S_{2})$. If, in addition, $M_{1}$ is complete, then $\sigma(S_{1}) \subset \sigma(S_{2})$ and $\sigma_{\ess}(S_{1}) \subset \sigma_{\ess}(S_{2})$.
\end{corollary}

\begin{proof}
Follows from Propositions \ref{Spectrum of Self-adj}, \ref{Inclusion of Spectra finite sheeted}, \ref{Comparison of Bottoms finite sheeted} and Corollary \ref{Inequality of bottoms}. \qed
\end{proof}
\medskip

The following characterization of the bottom of the essential spectrum of a Schr\"{o}dinger operator follows from the Decomposition Principle (\cite[Proposition 1]{MR1823312}) and Propositions \ref{Rayleigh} and \ref{Donelly}. Recall that this quantity is infinite when the spectrum is discrete.

\begin{proposition}[{\cite[Proposition 3.2]{MR2891739}}]\label{Bottom of Essential spectrum}
Let $S$ be a Schr\"{o}dinger operator on a complete manifold $M$ and let $(K_{k})_{k \in \mathbb{N}}$ be an exhausting sequence of $M$. Then
%Let $(U_{n})_{n \in \mathbb{N}}$ be an increasing sequence of open, precompact and smoothly bounded domains which cover $M$. Then the bottom of the essential spectrum of $S$ is given by
\[
\lambda_{0}^{\ess}(S) = \lim_{k} \lambda_{0}(S, M \smallsetminus K_{k}),
\]
where $\lambda_{0}(S, M \smallsetminus K_{k})$ is the bottom of the spectrum of $S$ on $M \smallsetminus K_{k}$.
\end{proposition}

%From this proposition, it is easy to see that if $(U_{n})_{n \in \mathbb{N}}$ is an increasing sequence of open, precompact and smoothly bounded domains of $M$, with $M=\cup_{n \in \mathbb{N}} U_{n}$, then the bottom of the essential spectrum of $S$ is the limit of the bottom of the spectrum of $S$ on $M \smallsetminus U_{n}$. From this and Corollary \ref{finite sheeted schroedinger} we obtain the following:

\begin{corollary}
Let $p \colon M_2 \to M_1$ be a finite sheeted Riemannian covering of complete manifolds.
%Let $S_{1}$ be a Schr\"{o}dinger operator on $M_{1}$ and $S_{2}$ its lift on $M_{2}$.
Consider a Schr\"{o}dinger operator $S_1$ on $M_{1}$ and its lift $S_{2}$ on $M_{2}$. 
Then $\lambda^{\ess}_{0}(S_1) = \lambda^{\ess}_0(S_2)$ and in particular, $\sigma_{\ess}(S_1) \neq \emptyset$ if and only if $\sigma_{\ess}(S_2) \neq \emptyset$. 
\end{corollary}

\begin{proof}
Follows from Corollary \ref{finite sheeted schroedinger} and Proposition \ref{Bottom of Essential spectrum}. \qed
\end{proof}

\section{Infinite deck transformations group}\label{high symmetry section}

%In this section, we prove a result which can be applied to more general situations than Riemannian coverings. 
Let $M$ be a Riemannian manifold, $E \to M$ a Riemannian or Hermitian vector bundle, endowed with a connection $\nabla$ and $D \colon \Gamma(E) \to \Gamma(E)$ a differential operator on $E$.
If $M$ has empty boundary, set $\mathcal{D}(D) := \Gamma_{c}(E)$. If $M$ has non-empty boundary, consider $$\mathcal{D}(D) := \{ \eta \in \Gamma_{c}(E) : a \nabla_{n}\eta + b \eta = 0 \text{ on } \partial M \},$$
where $n$ is the inward pointing normal to $\partial M$ and $a$, $b$ are real or complex valued functions (depending on whether $E$ is Riemannian or Hermitian) defined on $\partial M$. It is worth to point out that we do not impose any assumptions on $a$ and $b$. From Lemma \ref{closability}, the operator $D$ is closable and denote by $\overline{D}$ its closure. %The following theorem may be already known in some special cases.

\begin{theorem}\label{Equality of spectrum and essential spectrum}
Let $\Gamma$ be a group of automorphisms of $E$ preserving the metric of $E$, such that the induced action on $M$ is isometric and $D(g_{*}\eta) = g_{*} D\eta$, for any $g \in \Gamma$ and $\eta \in \Gamma(E)$. If $M$ has non-empty boundary, assume that $\nabla$, $a$ and $b$ are $\Gamma$-invariant along the boundary. %$\partial M$.
If for any compact $K \subset M$, there exists $g \in \Gamma$, such that $gK \cap K = \emptyset$,
then $\sigma_{\ap}(\overline{D}) = \sigma_{W}(\overline{D})$ and $\overline{D}$ does not have eigenvalues of finite multiplicity.
\end{theorem}

\begin{proof}
Let $\lambda \in \sigma_{\ap}(\overline{D})$. From Lemma \ref{Spectrum of Closure}, there exists $(\eta_{k})_{k \in \mathbb{N}} \subset \mathcal{D}(D)$, such that $\| \eta_{k} \|_{L^{2}(E)} = 1$ and $(D - \lambda) \eta_{k} \rightarrow 0$ in $L^{2}(E)$. Since $\eta_{k}$ is compactly supported, there exists an exhausting sequence $(K_{k})_{k \in \mathbb{N}}$ of $M$, such that $\supp \eta_{k} \subset K_{k}$, for all $k \in \mathbb{N}$. For any $k \in \mathbb{N}$, consider $g_{k} \in \Gamma$, such that $g_{k}K_{k} \cap K_{k} = \emptyset$, and set $\zeta_{k} := (g_{k})_{\star} \eta_{k}$. Then $\zeta_{k} \in \Gamma_{c}(E)$ and if $M$ has non-empty boundary, then $\zeta_{k}$ satisfies the same boundary conditions with $\eta_{k}$, since via isometries the boundary is mapped to itself and so does the inward pointing normal. It follows that $\zeta_{k} \in \mathcal{D}(D)$, $\| \zeta_{k} \|_{L^{2}(E)} = 1$ and $(D - \lambda) \zeta_{k} \rightarrow 0$ in $L^{2}(E)$. Clearly, $\supp \zeta_{k} = g_{k} (\supp \eta_{k})$, which yields that for any compact $K \subset M$, there exists $k_{0} \in \mathbb{N}$, such that $\supp \zeta_{k} \cap K = \emptyset$, for all $k \geq k_{0}$.
%$\supp \zeta_{n}$ eventually leave every compact subset of $M$. 
This implies that $\zeta_{k} \rightharpoonup 0$ in $L^{2}(E)$, that is, $(\zeta_{k})_{k \in \mathbb{N}}$ is a Weyl sequence for $\overline{D}$ and $\lambda$. Hence, $\lambda \in \sigma_{W}(\overline{D})$.

Assume that there exists an eigenvalue $\lambda$ of $\overline{D}$ of finite multiplicity, and consider $\theta \in \mathcal{D}(\overline{D})$ with $\| \theta \|_{L^{2}(E)}=1$ and $\overline{D} \theta = \lambda \theta$. Then there exists $(\eta_{k})_{k  \in \mathbb{N}} \subset \mathcal{D}(D)$, such that $\eta_k \rightarrow \theta$ and $D \eta_k \rightarrow \overline{D} \theta$. Clearly, for $g \in \Gamma$, we have $g_{*} \eta_{k} \in \mathcal{D}(D)$,%for any $n \in \mathbb{N}$, 
$g_{*} \eta_k \rightarrow g_{*}\theta$ and $D(g_{*} \eta_k) \rightarrow g_{*}(\overline{D} \theta)$, which yields that $g_{*}\theta \in \mathcal{D}(\overline{D})$ and $\overline{D}(g_{*} \theta) = \lambda (g_{*} \theta)$.

Let $(K_{k})_{k \in \mathbb{N}}$ be an exhausting sequence of $M$ and consider $(g_{k})_{k \in \mathbb{N}} \subset \Gamma$, such that $g_{k}K_{k} \cap K_{k} = \emptyset$, for any $k \in \mathbb{N}$. It is clear that the sections $\theta_{k} := (g_{k})_{*}\theta$ satisfy $\overline{D} \theta_k = \lambda \theta_k$ and $\| \theta_k \|_{L^{2}(E)} = 1$, for all $k \in \mathbb{N}$. Since the eigenspace corresponding to $\lambda$ is finite dimensional, after passing to a subsequence, we may assume that $\theta_k \rightarrow \theta_0$ in $L^{2}(E)$, for some $\theta_{0}$, with $\| \theta_0 \|_{L^{2}(E)}=1$. Consider a non-zero $\zeta \in \Gamma_c(E)$ and set $\zeta_{k} := (g_{k}^{-1})_{*}\zeta$. Then
\[
\langle \theta_k , \zeta \rangle_{L^{2}(E)}^{2} = \langle \theta, \zeta_k \rangle_{L^{2}(E)}^{2} \leq \| \zeta \|_{L^{2}(E)}^{2} \int_{\supp \zeta_k} \| \theta \|^{2}.
\]
%Since $\supp \zeta$ is eventually a subset of $K_{k}$, it follows that for any compact $K\subset M$, there exists $k_{0} \in \mathbb{N}$, such that $\supp \zeta_k \cap K = \emptyset$, for all $k \geq k_{0}$.
%eventually leave every compact subset of $M$. 
Let $\varepsilon > 0$ and consider a compact $K \subset M$, such that $\int_{M \smallsetminus K} \| \theta \|^{2} < \varepsilon^{2}/ \| \zeta \|_{L^{2}(E)}^{2}$. 
Since $\supp \zeta$ and $K$ are eventually subsets of $K_{k}$, there exists $k_{0} \in \mathbb{N}$, such that $\supp \zeta_k \cap K = \emptyset$, for all $k \geq k_{0}$.
Therefore, for $k \geq k_{0}$, we have $\supp \zeta_{k} \subset M \smallsetminus K$, and in particular, $|\langle \theta_k , \zeta \rangle_{L^{2}(E)}| < \varepsilon$. This yields that $\theta_{k} \rightharpoonup 0$ in $L^{2}(E)$, which is a contradiction, since $\theta_{k} \rightarrow \theta_0$ in $L^{2}(E)$ and $\| \theta_0 \|_{L^{2}(E)} = 1$. \qed
\end{proof}

\begin{theorem}\label{Equality of spectrum and ess spectrum Friedrichs}
Assume that $D$ is symmetric and bounded from below, and denote by $D^{(F)}$ its Friedrichs extension. Under the assumptions of Theorem \ref{Equality of spectrum and essential spectrum}, the spectrum of $D^{(F)}$ is essential and $D^{(F)}$ does not have eigenvalues of finite multiplicity.
\end{theorem}

\begin{proof}
Let $\eta \in \mathcal{D}(D^{(F)})$ and $g \in \Gamma$. From the invariance of $\mathcal{D}(D)$ and $D$ under the action of $\Gamma$, it follows that $g_{*}\eta \in \mathcal{D}(D^{(F)})$ and $D^{(F)}(g_{*}\eta) = g_{*}(D^{(F)} \eta)$. As in the proof of Theorem \ref{Equality of spectrum and essential spectrum}, it follows that $D^{(F)}$ does not have eigenvalues of finite multiplicity.
%Assume that there exists an eigenvalue $\lambda$ of $D^{(F)}$ of finite multiplicity and consider $\varphi \in \mathcal{D}(D^{(F)})$ with $\| \varphi \|_{L^{2}(E)} = 1$ and $D^{(F)} \varphi = \lambda \varphi$. Consider an exhausting sequence $(K_{n})_{n \in \mathbb{N}}$ of $M$ and let $(g_{n})_{n \in \mathbb{N}} \subset \Gamma$, such that $g_{n}K_{n} \cap K_{n} = \emptyset$ for any $n \in \mathbb{N}$. Consider the sections $\varphi_{n} := (g_{n})_{*}\varphi$, with $n \in \mathbb{N}$. As in the proof of Theorem \ref{Equality of spectrum and essential spectrum}, after passing to a subsequence, $\varphi_{n} \rightarrow \varphi_{0}$ for some $\varphi_{0}\in L^{2}(E)$ and $\varphi_{n} \rightharpoonup 0$ in $L^{2}(E)$, which is a contradiction. Therefore, $D^{(F)}$ does not have eigenvalues of finite multiplicity.
From Proposition \ref{Spectrum of Self-adj}, we obtain that $\sigma(D^{(F)}) = \sigma_{\ess}(D^{(F)})$.
\qed
\end{proof}\medskip

The above theorems can be applied to Riemannian coverings with infinite deck transformations group. In the context of the previous section, we obtain the following consequences.

\begin{corollary}\label{Infinite Deck Transformations}
If the deck transformations group of the covering is infinite, then $\overline{D}_{2}$ does not have eigenvalues of finite multiplicity and $\sigma_{\ap}(\overline{D}_{2}) = \sigma_{W}(\overline{D}_{2})$.
\end{corollary}

\begin{proof}
Follows immediately from Theorem \ref{Equality of spectrum and essential spectrum}, for $\Gamma$ being the deck transformations group of the covering. \qed
%, since the deck transformations group of the covering satisfies the assumptions of. \qed
\end{proof}\medskip

\noindent{\emph{Proof of Corollary \ref{Infinite Deck Transformations self-adj}:}} Follows from Corollary \ref{Infinite Deck Transformations} and Proposition \ref{Spectrum of Self-adj}. \qed \medskip

\noindent{\emph{Proof of Corollary \ref{Infinite Deck Transformations Friedrichs}}:} Follows from Theorem \ref{Equality of spectrum and ess spectrum Friedrichs}, for $\Gamma$ being the deck transformations group of the covering. \qed

\begin{corollary}
Let $M$ be a complete Riemannian manifold and assume that there exists a non-zero $\lambda_{0}(M)$-harmonic function in $L^{2}(M)$. If $\Gamma$ is a discrete group acting freely and properly discontinuously on $M$ via isometries, then $\Gamma$ is finite.
\end{corollary}

\begin{proof}
Follows from Corollary \ref{Infinite Deck Transformations self-adj}, since for any complete (and connected) Riemannian manifold $M$, the space of square-integrable, $\lambda_{0}(M)$-harmonic functions is either trivial or one dimensional. \qed
\end{proof} \medskip

Besides Riemannian coverings, the above theorems can be applied to manifolds with high symmetry. For instance, it follows that the spectrum of the Laplacian on a non-compact homogeneous space is essential. Moreover, we obtain the analogous statement, if there exists a non-compact Lie group acting on the manifold properly discontinuously via isometries.

\begin{comment}
Some other applications of these theorems are presented below.

\begin{corollary}
Let $M$ be a non-compact homogeneous space and $\Delta$ the Laplacian on $M$. Then $\Delta$ does not have eigenvalues of finite multiplicity and in particular, $\sigma(M) = \sigma_{ess}(M)$.
\end{corollary}

\begin{corollary}
Let $M$ be a Riemannian manifold with $\sigma(M) \neq \sigma_{ess}(M)$. If $\Gamma$ is a discrete group acting freely and properly discontinuously on $M$ via isometries, then $\Gamma$ is finite.
\end{corollary}
\end{comment}
%It is well-known that for any Riemannian manifold $M$, the space of square-integrable, $\lambda_{0}(M)$-harmonic functions is either trivial or one dimensional. Therefore, we obtain the following:

\section{Coverings of closed manifolds}\label{closed underlying section}

The \textit{Cheeger's constant} of a Riemannian manifold $M$ is defined by
%Let $M$ be a complete Riemannian manifold with (possibly empty) smooth boundary. Recall that the \textit{Cheeger's constant} is given by
\[
h(M) := \inf_{K} \frac{\text{Area}(\partial K)}{\Vol(K)},
\]
where the infimum is taken over all compact and smoothly bounded domains $K$ of $M$ which do not intersect $\partial M$.
%that divide $M$ in a bounded piece $\inte (S)$ that does not intersect $\partial M$ and another one.
It is related to $\lambda_{0}(M)$ via Cheeger's inequality (cf. \cite{MR0402831}):
\[
\lambda_{0}(M) \geq \frac{1}{4} h(M)^{2}.
\]
Brooks \cite{MR656213} actually proved that a normal Riemannian covering of a closed manifold is amenable if and only if the Cheeger's constant of the covering space is zero. The following result is an extension of that of Brooks, to not necessarily normal coverings.
%We extend this result to not necessarily normal coverings in the following theorem.
%As stated in the Introduction, we present an elementary proof for this, which extends it to not necessarily normal Riemannian coverings.

\begin{theorem}\label{result}
Let $p \colon M_{2} \to M_{1}$ be a Riemannian covering with $M_{1}$ closed. If $h(M_{2}) = 0$, then $p$ is amenable.
\end{theorem}

In order to prove this theorem, we need the following proposition.
%from the work of Buser on inverse Cheeger's inequality, 
In the sequel, for a subset $W$ of $M$, we denote by $B(W,r)$ the tubular neighborhood
%we denote by $B( W ,r)$ the tubular neighborhood of radius $r$ about a subset $W$ of $M$, that is, 
$$B(W,r) := \{ z \in M : d(z,W) < r \}.$$
%the set of points of $M$ with distance less than $r$ from $W$.

\begin{proposition}[{{\cite[Lemma 7.2]{Buser}}}]\label{Cutting hairs}
Let $M$ be a non-compact, complete Riemannian manifold, without boundary and with Ricci curvature bounded from below. Then there exists a constant $c$ depending only on the dimension of $M$, such that for any compact and smoothly bounded domain $K$ of $M$, with $\Area (\partial K) / \Vol (K) =: H$, and any $0< r \leq 1/2c \min\{ 1,1/H \}$,
%Let $A$ be a compact and smoothly bounded domain of $M$ with $\Area (\partial A) / \Vol (A) =: H$ and consider $0< r \leq 1/2c \min\{ 1,1/H \}$, where $c$ is a constant depending only on the dimension of $M$. Then 
there exists a bounded, open $U \subset M$, such that
\[
\frac{\Vol (B(\partial U , r))}{\Vol (U)} \leq C(r) H,
\]
where $C(r)$ is a constant depending on $r$, the dimension of $M$ and the lower bound of the Ricci curvature.
\end{proposition}

\begin{corollary}\label{Small Ratio}
Let $M$ be a non-compact, complete Riemannian manifold, without boundary and with Ricci curvature bounded from below. If $h(M) = 0$, then for any $\varepsilon,r>0$, there exists a bounded, open $U \subset M$, such that
\[
\frac{\Vol (B(\partial U, r))}{\Vol (U \smallsetminus B(\partial U, r))} < \varepsilon.
\]
\end{corollary}

\begin{proof}
Let $r > 0$ and $0 < r_{0} \leq 1/2c$, where $c$ is the constant from Proposition \ref{Cutting hairs}. Denote by $\mathfrak{g}$ the original Riemannian metric and consider the metric $\mathfrak{h} := C \mathfrak{g}$, where $C:= r_{0} /r$. For any compact and smoothly bounded domain $K$ of $M$, we have
\[
\frac{{\Area}_{\mathfrak{h}} (\partial K)}{{\Vol}_{\mathfrak{h}} (K)} = C^{-1/2} \frac{{\Area}_{\mathfrak{g}} (\partial K)}{{\Vol}_{\mathfrak{g}} (K)}.
\]
%Since $h(M, \mathfrak{g}) = 0$, it follows that $h(M, \mathfrak{h}) = 0$. 
Since the Cheeger's constant of $M$ with respect to $\mathfrak{g}$ is zero, it follows that so is the Cheeger's constant of $M$ with respect to $\mathfrak{h}$.
From Proposition \ref{Cutting hairs}, for any $\delta > 0$, there exists a bounded, open $U \subset M$, such that
\[
\frac{{\Vol}_{\mathfrak{h}} (B_{\mathfrak{h}}(\partial U , r_{0}))}{{\Vol}_{\mathfrak{h}} (U)} < \delta.
\]
It follows that
\[
\frac{{\Vol}_{\mathfrak{g}} (B_{\mathfrak{g}}(\partial U , r))}{{\Vol}_{\mathfrak{g}} (U)} = \frac{{\Vol}_{\mathfrak{h}} (B_{\mathfrak{g}}(\partial U , r))}{{\Vol}_{\mathfrak{h}} (U)} = \frac{{\Vol}_{\mathfrak{h}} (B_{\mathfrak{h}}(\partial U , r_{0}))}{{\Vol}_{\mathfrak{h}} (U)} < \delta.
\]
This completes the proof, 
since ${\Vol}_{\mathfrak{g}} (U) \leq {\Vol}_{\mathfrak{g}} (U \smallsetminus B_{\mathfrak{g}}(\partial U , r)) + {\Vol}_{\mathfrak{g}} (B_{\mathfrak{g}}( \partial U,r))$.\qed
%and $\delta > 0$ was arbitrary. \qed
\end{proof}
\medskip

\noindent{\emph{Proof of Theorem \ref{result}:}}
Evidently, if $M_{2}$ is closed, then $p$ is finite sheeted and in particular, amenable. Therefore, it remains to prove the statement for $M_{2}$ non-compact.
Consider the universal covering $p_{1} \colon \tilde{M} \to M_{1}$,
%Let $\tilde{M}$ be the simply connected covering space of $M_{i}$ and consider the coverings $p_{i} \colon \tilde{M} \to M_{i}$. 
fix $x \in M_{1}$, $u \in p_{1}^{-1}(x)$ and identify $\pi_{1}(M_{2}) \backslash \pi_{1}(M_{1})$ with $p^{-1}(x)$, as in the beginning of Subsection \ref{spectrum subsection}. Denote by $D_{y}$ the fundamental domain of $p$ centered at $y$, with $y \in p^{-1}(x)$. It is clear that for $y \in p^{-1}(x)$ and $z,w \in D_{y}$, we have 
\[d(z,w) \leq d(y,z) + d(y,w) = d(x,p(z)) + d (x,p(w)) \leq 2 \text{diam} (M_{1}),\]
which yields that $\text{diam}(D_{y}) \leq 2 \text{diam}(M_{1})$, for all $y \in p^{-1}(x)$. Let $r > 2 \text{diam}(M_{1})$ and %consider the finite set
\[
G_{r} := \{ g \in \pi_{1}(M_{1}) : d(u,gu) < r \}.
\]
%From Remark \ref{Orbits}, for any $y_{1}, y_{2} \in p^{-1}(x)$, we have that $d(y_{1},y_{2}) < r$ if and only if there exists $g \in G_{r}$, such that $y_{2} = y_{1} \cdot g$.
% if and only if $d(y_{1} , y_{2}) < r$.

From Corollary \ref{Small Ratio}, for any $\varepsilon > 0$, there exists a bounded, open $U \subset M_{2}$, such that
\begin{equation}\label{Volume Estimate}
\frac{\Vol (B(\partial U, 2r))}{\Vol(U\smallsetminus B(\partial U , 2r))} < \varepsilon.
\end{equation}
Consider the finite sets
\begin{eqnarray}
F &:=& \{ y \in p^{-1}(x) : y \in U \smallsetminus B(\partial U, r) \}, \nonumber \\
F^{\prime} &:=& \{ y \in p^{-1}(x) : y \in B(\partial U ,r) \}. \nonumber
\end{eqnarray}
Recall that $r > 2 \text{diam}(M_{1}) \geq \text{diam}(D_{y})$, for all $y \in p^{-1}(x)$, and $M_{2}$ is covered by the fundamental domains $D_{y}$, with $y \in p^{-1}(x)$. Evidently, $U \smallsetminus B(\partial U, 2r)$ is contained in the union of $D_{y}$, with $y \in F$. Furthermore, $B(\partial U, 2r)$ contains the union of $D_{y}$, with $y \in F^{\prime}$.
%Since $r > 2 \text{diam}(M_{1}) \geq \text{diam}(D_{y})$ for all $y \in p^{-1}(x)$ and $M_{2} = \cup_{y \in p^{-1}(x)} D_{y}$, it follows that $A \smallsetminus B(\partial A , 2r) \subset \cup_{y \in F} D_{y}$ and $B(\partial A , 2r) \supset \cup_{y \in Q} D_{y}$.
From (\ref{Volume Estimate}), since the intersection of different fundamental domains is of measure zero, and $\Vol(D_{y}) = \Vol(M_{1})$, for any $y \in p^{-1}(x)$, it follows that
\[
\frac{\#(F^{\prime})}{\#(F)} < \varepsilon.
\] 

Let $g \in G_{r}$ and $y \in F \smallsetminus Fg$. Then $y \in U$, $d (y, \partial U) \geq r$ and $y \cdot g^{-1}\notin F$. From Remark \ref{Orbits}, it follows that $d(y,y\cdot g^{-1}) < r$.
%Since $d(y, y \cdot g^{-1} ) < r$, 
Therefore, $y \cdot g^{-1} \in U$ and $d(y \cdot g^{-1}, \partial U) < r$, which yields that $y \cdot g^{-1} \in F^{\prime}$. Hence, $F \smallsetminus Fg \subset F^{\prime} g$ and in particular, we obtain that
\[
\#(F \smallsetminus Fg) \leq \#(F^{\prime}) < \varepsilon \#(F).
\]

For any finite $G \subset \pi_{1}(M_{1})$, there exists $r > 2 \text{diam}(M_{1})$, such that $G \subset G_{r}$. The above arguments imply that for any finite $G \subset \pi_{1}(M_{1})$ and $\varepsilon > 0$, there exists a F\o{}lner set for $G$ and $\varepsilon$. From Proposition \ref{Folner}, it follows that $p$ is amenable. \qed

\section{Applications and examples}\label{applications section}

Throughout most of this section we restrict ourselves to Schr\"{o}dinger operators and present some consequences of our main results in this context. We begin with some auxiliary considerations.

We first introduce the notion of renormalized Schr\"{o}dinger operators. This notion was introduced by Brooks in \cite{MR783536} for the Laplacian on complete manifolds without boundary. Let $S$ be a Schr\"{o}dinger operator on a possibly non-connected Riemannian manifold $M$ without boundary, and let $\varphi \in C^{\infty}(M)$ be a positive $\lambda$-eigenfunction of $S$. It is worth to point out that we do not require $\varphi$ to be square-integrable or $M$ to be complete. Consider the space
$
L^{2}_{\varphi}(M) := \{ [v] : v\varphi \in L^{2}(M) \},
$
where two functions are equivalent if they are almost everywhere equal, with the inner product %of this space is defined by
$
\langle v_{1} , v_{2} \rangle_{L^{2}_{\varphi}(M)} := \int_{M} v_{1}v_{2}\varphi^{2}.
$
Clearly, the map $\mu_{\varphi} \colon L^{2}_{\varphi}(M) \to L^{2}(M)$, given by $\mu_{\varphi}v := v \varphi $ is an isometric isomorphism, which yields that $L^{2}_{\varphi}(M)$ is a separable Hilbert space.

The \textit{renormalized Schr\"{o}dinger operator} $S_{\varphi} \colon \mathcal{D}(S_{\varphi}) \subset L^{2}_{\varphi}(M) \to L^{2}_{\varphi}(M)$ is defined by
$
S_{\varphi}v := \mu_{\varphi}^{-1}(S^{(F)} - \lambda)(\mu_{\varphi}v),
$
for all $v \in \mathcal{D}(S_{\varphi})$, where $\mathcal{D}(S_{\varphi}) := \mu_{\varphi}^{-1}(\mathcal{D}(S^{(F)}))$ and $S^{(F)}$ is the Friedrichs extension of $S$. Clearly, the following diagram is commutative
\begin{equation*} \label{diagram}
\begin{tikzpicture}
\matrix (m) [matrix of math nodes,row sep=2em,column sep=4em,minimum width=2em]
{
	\mathcal{D}(S_{\varphi}) & \mathcal{D}(S^{(F)}) \\
	L_{\varphi}^{2}(M) & L^{2}(M) \\};
\path[-stealth]
(m-1-1) edge node [left] {$S_{\varphi}$} (m-2-1)
(m-1-1.east|-m-1-2)edge node [above] {$\mu_{\varphi}$} node [below] {$\simeq$} (m-1-2)
(m-2-1.east|-m-2-2) edge node [below] {$\simeq$}
node [above] {$\mu_{\varphi}$} (m-2-2)
(m-1-2) edge node [right] {$S^{(F)} - \lambda$} (m-2-2);
\end{tikzpicture}
\end{equation*}
In particular, $S_{\varphi}$ is self-adjoint and $\sigma(S_{\varphi}) = \sigma(S) - \lambda$. From Proposition \ref{Rayleigh}, it follows that
\[
\lambda_{0}(S_{\varphi}) \leq \inf_{f \in C^{\infty}_{c}(M) \smallsetminus \{0\}}\mathcal{R}_{S_{\varphi}}(f) = \inf_{f \in C^{\infty}_{c}(M) \smallsetminus \{0\}} \frac{\langle S_{\varphi}f,f \rangle_{L^{2}_{\varphi}(M)}}{\| f \|_{L^{2}_{\varphi}(M)}^{2}},
\]
%It is well-known that there exists 
Let $(f_{k})_{k \in \mathbb{N}} \subset C^{\infty}_{c}(M) \smallsetminus \{0\}$, with $\mathcal{R}_{S}(f_{k}) \rightarrow \lambda_{0}(S)$.
Then, for $h_{k} := \mu_{\varphi}^{-1}(f_{k}) \in C^{\infty}_{c}(M)$, we have $\mathcal{R}_{S_{\varphi}}(h_{k}) \rightarrow \lambda_{0}(S_{\varphi})$.
Hence, the bottom of the spectrum of $S_{\varphi}$ can be approximated with Rayleigh quotients of compactly supported smooth functions in $M$. With a simple computation of the Rayleigh quotient of such a function (as in \cite[Section 2]{MR783536}, using the Divergence Theorem, instead of the $*$-operator), we obtain the following expression for $\lambda_{0}(S) - \lambda$.
\begin{proposition}\label{difference}
Let $S$ be a Schr\"{o}dinger operator on $M$ and let $\varphi \in C^{\infty}(M)$ be a positive $\lambda$-eigenfunction of $S$. Then
\[
\lambda_{0}(S) - \lambda = \inf_{f \in C_{c}^{\infty}(M) \smallsetminus \{0\}} \frac{\int_{M} \| \grad f\|^{2} \varphi^{2}}{ \int_{M} f^{2}\varphi^{2}}.
\]
\end{proposition}
The \textit{modified Cheeger's constant} of $M$ is defined by
\[
h_{\varphi}(M) := \inf_{K} \frac{\int_{\partial K} \varphi^{2}}{\int_{K} \varphi^{2}},
\]
where the infimum is taken over all compact and smoothly bounded domains $K$ of $M$. From the preceding proposition, it is easy to establish an analogue of Cheeger's inequality.
%closed (connected) hypersurfaces $S$ that divide a connected component of $M$ in a bounded piece $\inte (S)$ that does not intersect $\partial M$ and another one.
\begin{corollary}\label{renormalized 1}
Let $S$ be a Schr\"{o}dinger operator on $M$ and let $\varphi \in C^{\infty}(M)$ be a positive $\lambda$-eigenfunction of $S$. Then
\[
\lambda_{0}(S) - \lambda \geq \frac{1}{4} h_{\varphi}(M)^{2}.
\]
\end{corollary}
\begin{proof}
By virtue of Proposition \ref{difference}, the proof is the same as that of \cite[Lemma 3]{MR783536}.\qed
%Follows immediately from Proposition \ref{difference} and the proof of \cite[Lemma 3]{MR783536}, which works in our case as well. \qed
\end{proof}\medskip
Moreover, consider the quantity
\[
h_{\varphi}^{\ess}(M) := \sup_{K} h_{\varphi}(M \smallsetminus K),
\]
where the supremum is taken over all compact and smoothly bounded domains $K$ of $M$. For $\varphi = 1$, this quantity is denoted by $h^{\ess}(M)$.
%It is easy to see that if $(U_{n})_{n \in \mathbb{N}}$ is an increasing sequence of open, precompact and smoothly bounded domains of $M$, which cover $M$, then
%\[
%h_{\varphi}^{ess}(M) = \lim_{n} h_{\varphi}(M \smallsetminus U_{n}).
%\]

%From Proposition \ref{difference} and the proof of \cite[Lemma 3]{MR783536}, which works in our case as well, we obtain the following analogue of Cheeger's inequality.

%Recall that from Proposition \ref{Bottom of Essential spectrum}, the bottom of the essential spectrum of $S$ is given  by
%\[
%\lambda_{0}^{ess}(S) = \sup_{U} \lambda_{0}(S,M \smallsetminus U),
%\]
%where the supremum is taken over all open, precompact and smoothly bounded domains $U \subset M$. Consider a $\lambda$-eigenfunction $\varphi \in C^{\infty}(M)$ of $S$, which is positive in the interior of $M$. For an open, precompact and smoothly bounded domain $U \subset M$, consider the renormalized Schr\"{o}dinger operator on $M \smallsetminus U$. From Corollary \ref{renormalized 1}, taking the supremum over all such $U \subset M$, we obtain the following estimate.
\begin{corollary}\label{renormalized 2}
Let $S$ be a Schr\"{o}dinger operator on a complete manifold $M$ and consider a positive $\lambda$-eigenfunction $\varphi \in C^{\infty}(M)$ of $S$. Then
\[
\lambda_{0}^{\ess}(S) - \lambda \geq \frac{1}{4}h_{\varphi}^{\ess}(M)^{2}.
\]
\end{corollary}

\begin{proof}
Let $(K_{k})_{k \in \mathbb{N}}$ be an exhausting sequence of $M$, consisting of smoothly bounded domains. It is easy to see that
\[
h_{\varphi}^{\ess}(M) = \lim_{k} h_{\varphi}(M \smallsetminus K_{k}).
\]
From Corollary \ref{renormalized 1}, we have that \[\lambda_{0}(S,M \smallsetminus K_{k}) - \lambda \geq  \frac{1}{4} h_{\varphi}(M \smallsetminus K_{k}),\]
for any $k \in \mathbb{N}$. After taking the limit with respect to $k$, the statement follows from Proposition \ref{Bottom of Essential spectrum}. \qed
\end{proof}

%We are ready prove Theorem \ref{Improved Brooks}. 
%\medskip

\begin{remark}
The above arguments can be easily modified in order to obtain analogous results for manifolds with boundary. In that case, it suffices to consider a $\lambda$-eigenfunction of $S$ which is positive and smooth only in the interior of $M$. Then, in Proposition \ref{difference}, the infimum is taken over smooth functions with compact support in the interior of $M$.
\end{remark}

\noindent{\emph{Proof of Theorem \ref{Improved Brooks}:}} From Corollary \ref{incusion of spectra Schrodinger}, the first statement implies the second. From Corollary \ref{Inequality of bottoms}, the third statement follows from the second.

%If $\sigma(S_{1}) \subset \sigma_{ess}(S_{2})$, then $\lambda_{0}^{ess}(S_{2}) \leq \lambda_{0}(S_{1})$. Moreover, from Corollary \ref{Inequality of bottoms}, it follows that $\lambda_{0}(S_{1}) \leq \lambda_{0}(S_{2}) \leq \lambda_{0}^{ess}(S_{2})$.

Assume that $\lambda_{0}(S_{1}) = \lambda_{0}^{\ess}(S_{2})$, for some Schr\"{o}dinger operator $S_{1}$ on $M_{1}$. From Proposition \ref{Positive Spectrum}, there exists a positive $\lambda_{0}(S_{1})$-eigenfunction $\varphi \in C^{\infty}(M_{1})$ of $S_{1}$, and its lift $\hat{\varphi} \in C^{\infty}(M_{2})$ is a positive $\lambda_{0}(S_{1})$-eigenfunction of $S_{2}$. From Corollary \ref{renormalized 2}, it follows that $h_{\hat{\varphi}}^{\ess}(M_{2}) = 0$. Since $\varphi$ is positive and $M_{1}$ is closed, this yields that $h^{\ess}(M_{2}) = 0$.

Assume that $h^{\ess}(M_{2}) = 0$. Then $h(M_{2}) = 0$ and Theorem \ref{result} yields that $p$ is amenable. Assume that $p$ is finite sheeted. Then $M_{2}$ is closed. Consider a smoothly bounded domain $U$ of $M_{2}$, such that $M_{2} \smallsetminus U$ is connected. Evidently, $M_{2} \smallsetminus U$ is a compact manifold with boundary. From the definition of the Cheeger's constant, it is clear that $h(M_{2} \smallsetminus \overline{U}) = h(M_{2} \smallsetminus U)$. From \cite{MR0402831}, it follows that $h^{\ess}(M_{2}) \geq h(M_{2} \smallsetminus U) > 0$, which is a contradiction. Hence, $p$ is infinite sheeted. \qed
\medskip

For sake of completeness, we also prove the following corollary, describing the analogous properties for finite sheeted coverings.

\begin{corollary}
Let $p \colon M_{2} \to M_{1}$ be a Riemannian covering with $M_{1}$ closed. Let $S_{1}$ be a Schr\"{o}dinger operator on $M_{1}$ and $S_{2}$ its lift on $M_{2}$. Then the following are equivalent:
\begin{enumerate}[topsep=0pt,itemsep=-1pt,partopsep=1ex,parsep=0.5ex,leftmargin=*, label=(\roman*), align=left, labelsep=0em]
\item $p$ is finite sheeted,

\item $\sigma(S_{1}) \subset \sigma(S_{2})$ and $\sigma_{\ess}(S_{2}) = \emptyset$,

\item $\lambda_{0}(S_{1}) = \lambda_{0}(S_{2}) \notin \sigma_{\ess}(S_{2})$,

\item $h(M_{2}) = 0$ and $h^{\ess}(M_{2}) \neq 0$.
\end{enumerate}
\end{corollary}

\begin{proof}
If the covering is finite sheeted, the inclusion of spectra follows from Corollary \ref{finite sheeted schroedinger}. In this case, $M_{2}$ is closed, which yields that the spectrum of $S_{2}$ is discrete. From Corollary \ref{Inequality of bottoms}, the second statement implies the third.

Assume that the third statement holds. Since $\lambda_{0}(S_{1}) = \lambda_{0}(S_{2})$, as in the proof of Theorem \ref{Improved Brooks}, from Corollary \ref{renormalized 1}, it follows that $h(M_{2}) = 0$. From Theorem \ref{Improved Brooks}, it is clear that $h^{\ess}(M_{2}) \neq 0$.
%From Theorem \ref{result}, $p$ is amenable, and since $\lambda_{0}(S_{2}) \notin \sigma_{\ess}(S_{2})$, from Theorem \ref{Improved Brooks}, $p$ is finite sheeted.

%If $p$ is finite sheeted, then $M_{2}$ is closed, and as in the proof of Theorem \ref{Improved Brooks}, the fourth statement follows. Conversely, 
Assume that the fourth statement holds. Since $h(M_{2})=0$, from Theorem \ref{result}, $p$ is amenable. Since $h^{\ess}(M_{2}) \neq 0$, from Theorem \ref{Improved Brooks}, it follows that $p$ is finite sheeted.\qed
\end{proof}\medskip

The following characterization for points of the essential spectrum of a Schr\"{o}dinger operator is an immediate consequence of the Decomposition Principle.

\begin{proposition}\label{Zhislin}
Let $S$ be a Schr\"{o}dinger operator on a complete Riemannian manifold $M$ and let $\lambda \in \mathbb{R}$. Then $\lambda \in \sigma_{\ess}(S)$ if and only if there exists $(f_{k})_{k \in \mathbb{N}} \subset C^{\infty}_{c}(M)$, with $f_{k} = 0$ on $\partial M$, $\| f_{k} \|_{L^{2}(M)} = 1$, $(S - \lambda) f_{k} \rightarrow 0$ in $L^{2}(M)$ and for every compact $K \subset M$, there exists $k_{0} \in \mathbb{N}$, such that $\supp f_{k} \cap K = \emptyset$, for all $k \geq k_{0}$.
\end{proposition}

Our second application is motivated by Corollary 3.8 of the arXiv version of \cite{MBM}.%(which is not included in the published version).
%and its proof is closely modeled to the original one.

\begin{theorem}\label{Theorem Weaker Assuption}
Let $p \colon M_{2} \to M_{1}$ be a Riemannian covering with $M_{2}$ simply connected and complete. Let $S_{1}$ be a Schr\"{o}dinger operator on $M_{1}$ and $S_{2}$ its lift on $M_{2}$. If there exists a compact $K \subset M_{1}$, such that the image of the fundamental group of any connected component of $M_{1} \smallsetminus K$ in $\pi_{1}(M_{1})$ is amenable,
%$\pi_{1}(M_{1} \smallsetminus K)$ in $\pi_{1}(M_{1})$ is amenable, 
then $\sigma_{\ess}(S_{1}) \subset \sigma_{\ess}(S_{2})$.
\end{theorem}

\begin{proof}
Let $\lambda \in \sigma_{\ess}(S_{1})$. From Proposition \ref{Zhislin}, there exists $(f_{k})_{k \in \mathbb{N}} \subset C_{c}^{\infty}(M)$, such that $f_{k} = 0$ on $\partial M_{1}$, $\| f_{k} \|_{L^{2}(M_{1})} = 1$, $(S - \lambda) f_{k} \rightarrow 0$ in $L^{2}(M_{1})$ and for every compact $K_{0} \subset M_{1}$, there exists $k_{0} \in \mathbb{N}$, such that $\supp f_{k} \cap K_{0} = \emptyset$, for all $k \geq k_{0}$.
Without loss of generality, we may assume that the supports of $f_{k}$ are connected, since we may restrict each $f_{k}$ to a connected component of its support and obtain a sequence with the same properties.

Consider a compact $K \subset M_{1}$, such that the image of the fundamental group of any connected component of $M_{1} \smallsetminus K$ in $\pi_{1}(M_{1})$ is amenable. Clearly, after passing to a subsequence, we may assume that the functions $f_{k}$ are supported in $M_{1} \smallsetminus K$. Since for any $k \in \mathbb{N}$, the support of $f_{k}$ is connected, it follows that $\supp f_{k} \subset U_{k}$, where $U_{k}$ is a connected component of $M_{1} \smallsetminus K$.
From the Lifting Theorem, it follows that the inclusion $U_{k} \hookrightarrow M_{1}$ can be lifted to the covering space $M^{\prime}_{k} := M_{2}/\Gamma_{k}$, where $\Gamma_{k}$ is the image of $\pi_{1}(U_{k})$ in $\pi_{1}(M_{1})$. 
In particular, any $f_{k}$ can be lifted to some $f_{k}^{\prime} \in C^{\infty}_{c}(M^{\prime}_{k})$. 

Since the covering $q_{k} \colon M_{2} \to M^{\prime}_{k}$ is normal with deck transformations group $\Gamma_{k}$, it follows that it is amenable. If $q_{k}$ is finite sheeted, let $\tilde{f}_{k}$ be the normalized (in $L^{2}(M_{2})$) lift of $f_{k}^{\prime}$ on $M_{2}$. If $q_{k}$ is infinite sheeted, from Proposition \ref{Pull Up Lemma}, there exists $\tilde{f}_{k} \in C^{\infty}_{c}(M_{2})$, such that 
$\| \tilde{f}_{k} \|_{L^{2}(M_{2})}=1$, $\supp \tilde{f}_{k} \subset q_{k}^{-1}(\supp f^{\prime}_{k})$ and
$$\| (S_{2} - \lambda)\tilde{f}_{k} \|_{L^{2}(M_{2})} \leq \| (S_{k}^{\prime} - \lambda) f_{k}^{\prime} \|_{L^{2}(M_{k}^{\prime})}  + \frac{1}{k}  = \|(S_{1} - \lambda) f_{k} \|_{L^{2}(M_{1})} + \frac{1}{k},$$
where $S_{k}^{\prime}$ is the lift of $S_{1}$ on $M_{k}^{\prime}$. In particular, $(S_{2} - \lambda)\tilde{f}_{k} \rightarrow 0$ in $L^{2}(M_{2})$ and $\supp \tilde{f}_{k}$ is contained in $p^{-1}(\supp f_{k})$.
%In this way, we obtain a sequence $(\tilde{u}_{n})_{n \in \mathbb{N}} \subset C^{\infty}_{c}(M_{2})$, such that $\| \tilde{u}_{n} \|_{L^{2}(M_{2})} = 1$, $(S_{2} - \lambda)\tilde{u}_{n} \rightarrow 0$ in $L^{2}(M_{2})$ and $\supp \tilde{u}_{n}$ eventually leave every compact subset of $M_{2}$. 
From Proposition \ref{Zhislin}, it follows that $\lambda \in \sigma_{\ess}(S_{2})$.
%Clearly, $(u_{n}^{\prime})_{n \in \mathbb{N}}$ is a Weyl sequence for $S^{\prime}$ and $\lambda$, where $S^{\prime}$ is the lift of $S$ on $M^{\prime}$ and in particular, $\lambda \in \sigma_{ess}(S^{\prime})$. Note that the Riemannian covering $q : M_{2} \to M^{\prime}$ is normal with deck transformations group $\Gamma$. Since $\Gamma$ is amenable, from Theorem \ref{Inclusion of Spectrums} or Lemma \ref{Inclusion of Spectra finite sheeted} (depending on whether $q$ is infinite or finite sheeted), it follows that $\sigma_{ess}(\overline{S}) \subset \sigma_{ess}(\tilde{S})$. In particular, this yields that $\lambda \in \sigma_{ess}(\tilde{S})$.
\qed
\end{proof}

\begin{remark}
In the proof of Theorem \ref{Theorem Weaker Assuption}, the only properties of Schr\"{o}dinger operators used are essential self-adjointness and Proposition \ref{Zhislin}, which follows from the Decomposition Principle.
Therefore, this proof establishes the analogous result for 
%a more general class of differential operators. 
%More precisely, it suffices to restrict ourselves to 
essentially self-adjoint differential operators, for which the Decomposition Principle holds (cf. \cite{MR1823312}).
For instance, if $M_{1}$ has empty boundary, then the statement of Theorem \ref{Theorem Weaker Assuption} holds for any elliptic differential operator $D_{1}$, such that $D_{1}$ and $D_{2}$ are essentially self-adjoint on the spaces of compactly supported smooth sections.
\end{remark}

%Note that this proof establishes the analogous result for a more general class of differential operators. More precisely, it suffices to restrict ourselves to essentially self-adjoint differential operators for which the Decomposition Principle holds (cf. \cite{MR1823312}).
%For instance, if $M_{1}$ has empty boundary, then Theorem \ref{Theorem Weaker Assuption} holds for any elliptic differential operator $D_{1}$, such that $D_{1}$ and $D_{2}$ are essentially self-adjoint on the spaces of compactly supported smooth sections. \medskip

\noindent{\emph{Proof of Corollary \ref{Weaker assumption}:}} Follows immediately from Theorem \ref{Theorem Weaker Assuption} and Corollary \ref{Inequality of bottoms}. \qed
\medskip

Let $p \colon M_{2} \to M_{1}$ be a Riemannian covering of complete manifolds. As stated in the Introduction, there are examples where $p$ is non-amenable and $\lambda_{0}(M_{1}) = \lambda_{0}(M_{2})$. From Theorem \ref{Inclusion of Spectrums self-adj}, Propositions \ref{Inclusion of Spectra finite sheeted} and \ref{Spectrum of Self-adj}, if $p$ is amenable, then $\sigma(M_{1}) \subset \sigma(M_{2})$. It is natural to examine if this inclusion implies amenability of the covering. From Theorem \ref{Theorem Weaker Assuption}, it is easy to construct an example of a non-amenable, normal Riemannian covering $p \colon M_{2} \to M_{1}$ with $M_{1}$ complete, with bounded geometry and of finite topological type (that is, $M_{1}$ admits a finite triangulation, where the simplices are defined on the standard simplex with possibly some lower dimensional faces removed), such that $\sigma(M_{1})=\sigma(M_{2})$.

\begin{example}
Let $M_{1}$ be a $2$-dimensional torus with a cusp, endowed with a Riemannian metric, such that $M_{1}$ is complete and outside a compact set the metric is the standard metric of the flat cylinder. It is clear that $M_{1}$ is of finite topological type and has bounded geometry. From \cite[Theorem 1]{Lu}, it follows that $\sigma_{\ess}(M_{1}) = [0,+\infty)$. Clearly, there exists a compact subset $K$ of $M_{1}$, such that $\pi_{1}(M_{1} \smallsetminus K) = \mathbb{Z}$. From Theorem \ref{Theorem Weaker Assuption}, it follows that for the simply connected covering space $M_{2}$ of $M_{1}$, we have $\sigma_{\ess}(M_{2}) = [0,+\infty)$. However, $\pi_{1}(M_{1})$ is the free group in two generators, which is non-amenable (cf. \cite[Section 2]{MR3104995}).
\end{example}

The next simple observation, provides a sufficient geometric condition for amenability of coverings.

\begin{proposition}
Let $M_{1}$ be a complete Riemannian manifold, without boundary and with non-negative Ricci curvature. Then any covering $p \colon M_2 \to  M_1$ is amenable.
\end{proposition}

\begin{proof}
Let $\tilde{M}$ be the simply connected covering space of $M_{1}$. From the Bishop-Gromov Comparison Theorem, it follows that $\tilde{M}$ has polynomial growth and hence, every finitely generated subgroup of $\pi_{1}(M_{1})$ has polynomial growth (cf. \cite{MR0232311}). From Corollary \ref{subexp growth}, it follows that every finitely generated subgroup of $\pi_{1}(M_{1})$ is amenable and Corollary \ref{finitely generated subgroups} yields that so is $\pi_{1}(M_{1})$. Therefore, any covering $p \colon M_{2} \to M_{1}$ is amenable. \qed
\end{proof}
\medskip

%This proposition combined with our results for the Laplacian is trivial, since for a non-compact, complete Riemannian manifold $M$, without boundary and with non-negative Ricci curvature, we have $\sigma_{\ess}(M) = [0, + \infty)$. However, for more general differential operators, the previous proposition with our results, give non-trivial applications.

Next, we present an example of an infinite sheeted amenable covering with trivial deck transformations group. In particular, this implies that the results of Section \ref{high symmetry section} cannot be applied to arbitrary infinite sheeted amenable coverings.

\begin{example}\label{infinite sheeted amenable trivial dtg}
Let $\Gamma_{1}$ be the countable group of invertible, upper triangular $2 \times 2$ matrices with entries in $\mathbb{Q}$ and let $M_{1}$ be a Riemannian manifold with $\pi_{1}(M_{1}) = \Gamma_{1}$ (cf. \cite[Section 5]{BMP}). Let $\Gamma_{2}$ be the subgroup of $\Gamma_{1}$ consisting of diagonal matrices. Denote by $\tilde{M}$ the simply connected covering space of $M_{1}$ and consider $M_{2} := \tilde{M}/\Gamma_{2}$. It is easy to see that the covering $p \colon M_{2} \to M_{1}$ is infinite sheeted and does not have non-trivial deck transformations. However, $\Gamma_{1}$ is solvable and in particular, amenable (from Corollary \ref{solvable}), which yields that $p$ is an amenable covering.
\end{example}

Recall that in our main results there are no assumptions on the vector bundles, the connections and the differential operators. We end this section with an example which demonstrates that these play a crucial role in the behavior of the spectrum even under finite sheeted coverings. Namely, this example shows that whether or not the bottom of the spectrum of the connection Laplacian is preserved under a Riemannian covering depends on the corresponding metric connection.

If $M$ is a closed Riemannian manifold and $E \to M$ is a Riemannian vector bundle endowed with a metric connection $\nabla$, then the (corresponding) connection Laplacian is given by $\Delta = \nabla^{*}\nabla$. It is well-known that $\Delta \colon \Gamma(E) \subset L^{2}(E) \to L^{2}(E)$ is essentially self-adjoint and its spectrum is discrete (cf. \cite{MR1031992}).

\begin{example}
Consider $S^{1} := \mathbb{R}/\mathbb{Z}$ and the trivial bundle $E_{1} := S^{1} \times \mathbb{R}^{2}$ with the standard metric. We can identify smooth sections of $E_{1}$ with smooth maps $f \colon \mathbb{R} \to \mathbb{R}^{2}$ with $f(x) = f(x+1)$, for all $x \in \mathbb{R}$. For $\phi \in \mathbb{R}$, consider the metric connection $\nabla^{\phi}$, defined by
\[
\nabla^{\phi}_{\frac{d}{dx}}f(x) := \begin{pmatrix} \cos(x\phi) & -\sin(x\phi) \\
                \sin(x\phi) & \cos(x\phi) \end{pmatrix}  \frac{d}{dx} \begin{pmatrix} \cos(x\phi) & \sin(x\phi) \\
                -\sin(x\phi) & \cos(x\phi) \end{pmatrix}
  \begin{pmatrix} f_1(x) \\ f_2(x) \end{pmatrix},
\]
for any smooth section $f=(f_{1},f_{2})$ of $E_{1}$. Since the spectrum of the connection Laplacian $\Delta^{\phi}$ is discrete for any $\phi \in \mathbb{R}$, it is clear that $\lambda_{0}(\Delta^{\phi}, E_{1}) = 0$ if and only if there exists a parallel section of $E_{1}$ with respect to $\nabla^{\phi}$, or equivalently,
%It is easy to see that this happens if and only if 
$\phi = 2k \pi$, for some $k \in \mathbb{Z}$.

For $q \in \mathbb{N} \smallsetminus \{1\}$, consider a $q$-sheeted Riemannian covering of $S_{1}$ and the pullback bundle $E_{2}$ of $E_{1}$ endowed with the standard metric and the pullback connection $\nabla^{\phi}$. It is clear that $\lambda_{0}(\Delta^{2\pi},E_{2}) = \lambda_{0}(\Delta^{2\pi},E_{1})=0$. However, the above arguments imply that $\lambda_{0}(\Delta^{2\pi/q},E_{2}) = 0 < \lambda_{0}(\Delta^{2\pi/q},E_{1})$.
\end{example}

\noindent Institut f\"{u}r Mathematik, Humboldt-Universit\"{a}t zu Berlin \\
Unter den Linden 6, 10099 Berlin \\
E-mail address: polymerp@hu-berlin.de


\begin{bibdiv}
\begin{biblist}

\bib{MBM}{article}{
   author={Ballmann, Werner},
   author={Matthiesen, Henrik},
   author={Mondal, Sugata},
   title={Small eigenvalues of surfaces of finite type},
   journal={Compos. Math.},
   volume={153},
   date={2017},
   number={8},
   pages={1747--1768},
}

\bib{BMP}{article}{
	author={Ballmann, Werner},
	author={Matthiesen, Henrik},
	author={Polymerakis, Panagiotis}
	title={On the bottom of spectra under coverings},
	journal={Math. Z.},
	date={2017},
	DOI={10.1007/s00209-017-1925-9},
}

\bib{MR1823312}{article}{
   author={B\"ar, Christian},
   title={The Dirac operator on hyperbolic manifolds of finite volume},
   journal={J. Differential Geom.},
   volume={54},
   date={2000},
   number={3},
   pages={439--488},
}

\bib{MR3104995}{article}{
   author={B{\'e}rard, Pierre},
   author={Castillon, Philippe},
   title={Spectral positivity and Riemannian coverings},
   journal={Bull. Lond. Math. Soc.},
   volume={45},
   date={2013},
   number={5},
   pages={1041--1048},
}

\bib{MR2891739}{article}{
   author={Bessa, G. Pacelli},
   author={Montenegro, J. Fabio},
   author={Piccione, Paolo},
   title={Riemannian submersions with discrete spectrum},
   journal={J. Geom. Anal.},
   volume={22},
   date={2012},
   number={2},
   pages={603--620},
}

%\bib{MR638814}{article}{
%   author={Brooks, Robert},
%   title={A relation between growth and the spectrum of the Laplacian},
%   journal={Math. Z.},
%   volume={178},
%   date={1981},
%   number={4},
%   pages={501--508},
%}

\bib{MR783536}{article}{
   author={Brooks, Robert},
   title={The bottom of the spectrum of a Riemannian covering},
   journal={J. Reine Angew. Math.},
   volume={357},
   date={1985},
   pages={101--114},
}

\bib{MR656213}{article}{
	label={Bro80}
   author={Brooks, Robert},
   title={The fundamental group and the spectrum of the Laplacian},
   journal={Comment. Math. Helv.},
   volume={56},
   date={1981},
   number={4},
   pages={581--598},
}

\bib{Buser}{article}{
   author={Buser, Peter},
   title={A note on the isoperimetric constant},
   journal={Ann. Sci. \'Ecole Norm. Sup. (4)},
   volume={15},
   date={1982},
   number={2},
   pages={213--230},
}

\bib{MR0402831}{article}{
   author={Cheeger, Jeff},
   title={A lower bound for the smallest eigenvalue of the Laplacian},
   conference={
      title={Problems in analysis},
      address={Papers dedicated to Salomon Bochner},
      date={1969},
   },
   book={
      publisher={Princeton Univ. Press, Princeton, N. J.},
   },
   date={1970},
   pages={195--199},
}

\bib{MR0385749}{article}{
	author={Cheng, S. Y.},
	author={Yau, S. T.},
	title={Differential equations on Riemannian manifolds and their geometric		applications},
	journal={Comm. Pure Appl. Math.},
	volume={28},
	date={1975},
	number={3},
	pages={333--354},
}

\bib{MR592568}{article}{
   author={Donnelly, Harold},
   title={On the essential spectrum of a complete Riemannian manifold},
   journal={Topology},
   volume={20},
   date={1981},
   number={1},
   pages={1--14},
}

\bib{MR562550}{article}{
   author={Fischer-Colbrie, Doris},
   author={Schoen, Richard},
   title={The structure of complete stable minimal surfaces in $3$-manifolds
   of nonnegative scalar curvature},
   journal={Comm. Pure Appl. Math.},
   volume={33},
   date={1980},
   number={2},
   pages={199--211},
}

\bib{MR1361167}{book}{
   author={Hislop, P. D.},
   author={Sigal, I. M.},
   title={Introduction to spectral theory},
   series={Applied Mathematical Sciences},
   volume={113},
   note={With applications to Schr\"odinger operators},
   publisher={Springer-Verlag, New York},
   date={1996},
   pages={x+337},
}

\bib{MR1031992}{book}{
   author={Lawson, H. Blaine, Jr.},
   author={Michelsohn, Marie-Louise},
   title={Spin geometry},
   series={Princeton Mathematical Series},
   volume={38},
   publisher={Princeton University Press, Princeton, NJ},
   date={1989},
   pages={xii+427},
}

\bib{Lu}{article}{
   author={Zhiqin Lu},
   author={Detang Zhou},
   title={On the essential spectrum of complete non-compact manifolds},
   journal={J. Funct. Anal.},
   volume={260},
   date={2011},
   number={11},
   pages={3283--3298},
}

\bib{MR0232311}{article}{
   author={Milnor, J.},
   title={A note on curvature and fundamental group},
   journal={J. Differential Geom.},
   volume={2},
   date={1968},
   pages={1--7},
}

\bib{MR882827}{article}{
   author={Sullivan, Dennis},
   title={Related aspects of positivity in Riemannian geometry},
   journal={J. Differential Geom.},
   volume={25},
   date={1987},
   number={3},
   pages={327--351},
}

\bib{MR2744150}{book}{
   author={Taylor, Michael E.},
   title={Partial differential equations I. Basic theory},
   series={Applied Mathematical Sciences},
   volume={115},
   edition={2},
   publisher={Springer, New York},
   date={2011},
   pages={xxii+654},
   isbn={978-1-4419-7054-1},
}

\end{biblist}
\end{bibdiv}
\end{document}